\documentclass[11pt,a4paper]{article}

\usepackage{amsmath,amsfonts,amssymb,latexsym,graphics,epsfig,url}
\usepackage[dvipsnames]{xcolor}
\usepackage[normalem]{ulem}
\usepackage{amsthm}
\usepackage[english]{babel}
\usepackage{graphicx,tikz}
\usetikzlibrary{decorations.markings}
\usepackage{soul}
\usepackage[utf8]{inputenc}
\usepackage[T1]{fontenc}
\usepackage{hyperref} 
\hypersetup{colorlinks,
    linkcolor=blue, 
    anchorcolor=blue,
    citecolor=blue}
\tikzstyle{directed}=[postaction={decorate,
	decoration={markings,mark=at position 0.6 with {\arrow{latex}}}
}]
\tikzstyle{vertex}=[circle, draw, inner sep=0pt, minimum size=3pt]
\newcommand{\vertex}{\node[vertex]}

\newtheorem{theorem}{Theorem}[section]
\newtheorem{lemma}[theorem]{Lemma}

\newtheorem{proposition}[theorem]{Proposition}

\newtheorem{example}[theorem]{Example}

\def\r{\mbox{\boldmath $r$}}

\def\vec0{\mbox{\boldmath $0$}}

\def\AA{\mbox{$\mathcal{A}$}}

\begin{document}

\title{On Generalized Token Graphs
\thanks{The research of the first and last authors has been supported by National Natural Science Foundation of China (No.~12471334, No.~12131013), and Shaanxi Fundamental Science Research Project for Mathematics and Physics (No. 22JSZ009). The research of the second and third authors was funded by AGAUR from the Catalan Government under project 2021SGR00434 and MICINN from the Spanish Government under project PID2020-115442RB-I00.
The third author's research is also supported by a grant from the Universitat Polit\`ecnica de Catalunya, reference AGRUPS-2024.
The fourth author's research is 
supported by grants 
PID2023-150725NB-I00 funded by MICIU/AEI/10.13039/501100011033PID2023-150725NB-I00 and Gen. Cat. DGR 2017SGR1336.
}}

\author{Xiaodi Song$^{a,b}$, Cristina Dalf\'o$^b$,  
Miquel \`Angel Fiol$^c$, Merc\`e Mora$^c$, Shenggui Zhang$^{a}$\\ 
$^a${\small School of Mathematics and Statistics, Northwestern Polytechnical University}\\
    {\small Xi'an-Budapest Joint Research Center for Combinatorics, Northwestern Polytechnical University} \\
    {\small Xi'an, Shaanxi, P.R. China, {\tt songxd@mail.nwpu.edu.cn}, 
    {\tt sgzhang@nwpu.edu.cn}}\\
$^b${\small Departament de Matem\`atica, Universitat de Lleida} \\
		{\small Igualada (Barcelona), Catalonia, {\tt cristina.dalfo@udl.cat}}\\
$^c${\small Departament de Matem\`atiques, Universitat Polit\`ecnica de Catalunya} \\
    	{\small Barcelona Graduate School, Institut de Matem\`atiques de la UPC-BarcelonaTech (IMTech)}\\
    	{\small Barcelona, Catalonia, {\tt miguel.angel.fiol@upc.edu}, 
        {\tt merce.mora@upc.edu}}
}
\date{}
\maketitle

\begin{abstract}
The vertices of a $k$-token graph of a graph $G$ correspond to $k$ indistinguishable tokens placed on $k$ different vertices of $G$. Changing some conditions on both the nature of the tokens and the number of tokens allowed in each vertex of $G$, we define a generalization of token graphs, which we call generalized token graphs or simply supertoken graphs, which have different applications. Depending on the above conditions, different families of graphs (such as the Cartesian $k$-th power of $G$ by itself) are obtained, and we present some of their properties, including order, size, and connectivity. 
\end{abstract}



\section{Definition and particular cases}

Given some integers $n,k$, we denote by $C^n_k$ (with $k\le n$) and $C\!R^n_k$, respectively, the sets of combinations and multisets with repetition of $n$ elements taken $k$ at a time. Recall that the cardinalities of both sets are
$|C^n_k|=\binom{n}{k}$ and $|C\!R^n_k|=\binom{n+k-1}{k}$.

Let $G=(V,E)$ be a graph on $n$ vertices. Let $k$ and $s$ be positive integers such that $s\le k$.
Then,  the {\it $k$-supertoken} graph $F_{k}^s(G)$ [or $F^s_{k\times 1}(G)$] is the graph in which each vertex corresponds to a distribution of $k$ equal [or different] tokens between the $n$ vertices of $G$, and in such a way that no vertex can receive more than $s$ tokens.
Thus, every vertex of the $k$-supertoken graph corresponds to a multiset of unordered [or ordered] symbols representing the vertices of $G$.
Moreover, two vertices of $F_{k}^s(G)$ [or $F^s_{k\times 1}(G)$] are adjacent if the symmetric difference of their corresponding multisets 
is a pair of adjacent vertices $u,v\in V$. In other words, each vertex of the $k$-supertoken graph is defined by the multiset of vertices of $G$ having a token, and we can go from a vertex to an adjacent one in $F_{k}^s(G)$ [or $F^s_{k\times 1}(G)$] by simply moving, in $G$, one token to a feasible adjacent vertex.
In Figure~\ref{C_4+2-supertokens(C_4)}$(a)$-$(e)$ there are different 2-supertokens graphs of the cycle $G=C_4$ ($n=4$).
\begin{figure}[t]
    \begin{center}
  \includegraphics[width=15cm]{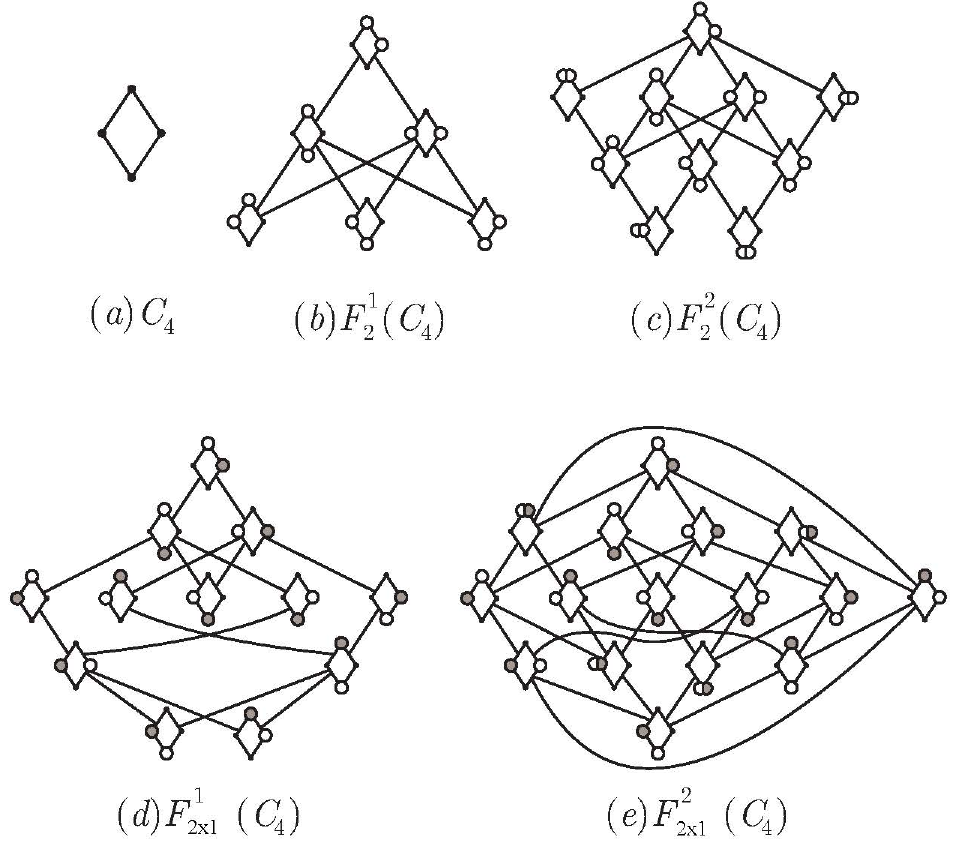}
  \end{center}
  \vskip-10.75cm
  \caption{$(a)$ The graph $C_4$; $(b)$ The supertoken graph $F_2^1(C_4)=F_2(C_4)\cong K_{2,4}\subset F_2^2(C_4)$; $(c)$ The supertoken graph $F_2^2(C_4)$; $(d)$ The supertoken graph $F_{2\times 1}^1(C_4)\subset F_{2\times 1}^2(C_4)$; $(e)$ The supertoken graph $F_{2\times 1}^2(C_4)\cong C_4\Box C_4$. The tokens are white or gray, and the rhombuses in $(b)$--$(e)$ represent the vertices of the supertoken graphs.}
   \label{C_4+2-supertokens(C_4)}
  \end{figure}

The supertoken graphs have been widely applied in various fields involving computer science, physics, chemistry, and so on. 
For example, moving tokens along the edges of a given graph to reach a final configuration is a class of reconfiguration problems in computer science; see Bonnet, Miltzow, and Rz\k{a}\.{z}ewski \cite{BMR2018}.  Moreover, the colored token problem is attributed to
Yamanaka, Horiyama, Mark Keil, Kirkpatrick, Otachi, Saitoh, Uehara, and Uno
\cite{yhmkosuu18}. See also Yamanaka, Demaine, Ito, Kawahara, Kiyomi, Okamoto, Saitoh, Suzuki, Uchizawa, and Uno \cite{ydikkossuu15}.
In physical and chemical applications, a class of supertoken graph is related to the exchange Hamiltonians in quantum mechanics (see Audenaert, Godsil, Royle, and Rudolph \cite{agrr07}). Besides, the minimum cycle basis construction of a class of supertoken graphs may be used to confirm that state-dependent coupling of automata in such a way that it does not violate the principle of microscopic reversibility, see Hammack and Smith \cite{HaSm17}.

Let us begin by giving the different kinds of 
$k$-supertoken graphs of a graph $G$ (with $n$ vertices and $m$ edges),
indicating the section in which they are dealt with (see a scheme in Table~\ref{taula:casos-part}):

\begin{table}[t]
\begin{center}
\begin{tabular}{|c||c|c|c|c|c|}
  \hline
  $s\setminus$ {\rm tokens} & all equal & all different                                      \\[.1cm]
  \hline\hline
  $1$              & $F_k^1(G)=F_k(G)$       & $F_{k\times 1}^1(G)$                                   \\[.1cm]
                   & Token graph or symmetric power           &                                                           \\[.1cm]
  \hline
  $1<s<k$          & $F_k^s(G)$          & $F_{k\times 1}^s(G)$                                   \\[.1cm]
  \hline
  $k$              & $F_k^k(G)$          & $F_{k\times 1}^k(G)=G\Box\stackrel{(k)}{\cdots}\Box G$ \\[.1cm]
                   & Reduced power                            & Cartesian product                                          \\[.1cm]
  \hline
\end{tabular}
\caption{Particular cases of the $k$-supertoken graph. 
}
\label{taula:casos-part}
\end{center}
\end{table}

\begin{itemize}
\item[$\circ$] The $k$ tokens are indistinguishable, and the maximum number of tokens per vertex is one:
$F_k^1(G)=F_k(G)$ (Section~\ref{sec:token}).

\item[$\circ$] The $k$ tokens are indistinguishable, and the maximum number of tokens per vertex is $s$, with $1<s<k$:
$F_k^s(G)$ (Section~\ref{sec:F_k^s(G)}).

\item[$\circ$] The $k$ tokens are indistinguishable, and the maximum number of tokens per vertex is $k$:
$F_k^k(G)$ (Section~\ref{sec:F_k^k(G)}).

\item[$\circ$] The $k$ tokens are all different, and the maximum number of tokens per vertex is one:
$F^{1}_{k\times 1}(G)$ (Section~\ref{sec:k-colors-1tokenvertex}).

\item[$\circ$] The $k$ tokens are all different, and the maximum number of tokens per vertex is $s$, with $1<s<k$:
$F^{s}_{k\times 1}(G)$ (Section~\ref{sec:F_k^s(G)-colors-dif}).

\item[$\circ$] The $k$ tokens are all different, and the maximum number of tokens per vertex is $k$:
$F_{k\times 1}^k(G)=G\Box\stackrel{(k)}{\cdots}\Box G$, which is the Cartesian $k$-th product of $G$ (Section~\ref{sec:prod-cartesia}).
\end{itemize}

When the tokens are indistinguishable, we say that they have the same color. Otherwise, we consider that each token has a different color from the others.

\section{Generalized token graphs}

In this section, we deal with results that can be applied to all generalized token graphs. For a graph $G=(V,E)$, define an unordered $k$-tuple (or an ordered $k$-tuple) $(\alpha_1,\alpha_2,\ldots,\alpha_k)$ to be a vertex of the supertoken graph $F_k^s(G)$ (or $F^{s}_{k\times 1}(G)$, respectively) with $\alpha_{i}\in V$ satisfying the number of times a vertex of $G$ appears with a token in $\{\alpha_1,\alpha_2,\ldots,\alpha_k\}$ is less than $s+1$.

Observe that 
\begin{equation}\label{eq:subset1}
    F^1_k(G)\subseteq F_k^2(G)\subseteq \cdots \subseteq F_k^k(G)
\end{equation}
and 
\begin{equation}\label{eq:subset2}
    F^{1}_{k\times 1}(G)\subseteq F^{2}_{k\times 1}(G)\subseteq \cdots\subseteq F^{k}_{k\times 1}(G).
\end{equation}
Barik and Verma \cite{bv24} showed that if $G$ is a connected graph, then, for $2\le k\le \frac{n}{2}$, $F_k(G)$ cannot be a tree for $n\ge 4$ and contains at least two cycles for $n\ge5$. Then, we get the following proposition from (\ref{eq:subset1}) and (\ref{eq:subset2}).
\begin{proposition}
    Let $F_{k}^s(G)$ and $F^{s}_{k\times 1}(G)$ be two $k$-supertoken graphs of a connected graph $G$ with order $n$, where $1\le s\le k$. Then, the following statements hold.
    \begin{itemize}
        \item [$(i)$] For $2\le k\le \frac{n}{2}$, the $k$-supertoken graph $F_{k}^s(G)$ cannot be a tree for $n\ge4$;
        \item [$(ii)$] For $2\le k\le n-2$, the $k$-supertoken graph $F^{s}_{k\times 1}(G)$ cannot be a tree for $n\ge4$.
    \end{itemize}
\end{proposition}
\begin{proof}
It suffices to prove that $F_{k}^s(G)$ and $F^{s}_{k\times 1}(G)$ contain at least one cycle. 
Since $F^{1}_{k}(G)=F_k(G)$ cannot be a tree for $n\ge4$, it follows from (\ref{eq:subset1}) that $F^s_k(G)$ cannot be a tree for $1\le s\le k$ and $2\le k\le \frac{n}{2}$. 
Next, consider $F^{1}_{k\times 1}(G)$ for $2\le k\le n-2$. Suppose first that the maximum degree $\Delta$ of $G$ is at least $3$. Let $u_0$ be a vertex with degree $\Delta(G)$, and $u_1,u_2,u_3\in N_G(u_0)$,  neighbors of $u_0$ in $G$. Let $\alpha_1,\ldots,\alpha_{k-2}\in V(G)\backslash\{u_0,u_1,u_2,u_3\}$. 
Note that 
\begin{align*}
    &(u_0,u_1,\alpha_1,\ldots,\alpha_{k-2}),~(u_2,u_1,\alpha_1,\ldots,\alpha_{k-2}),~(u_2,u_0,\alpha_1,\ldots,\alpha_{k-2}),~(u_2,u_3,\alpha_1,\ldots,\alpha_{k-2}),~\\
&(u_0,u_3,\alpha_1,\ldots,\alpha_{k-2}),~
(u_1,u_3,\alpha_1,\ldots,\alpha_{k-2}),~
(u_1,u_0,\alpha_1,\ldots,\alpha_{k-2}),~
(u_1,u_2,\alpha_1,\ldots,\alpha_{k-2}),~\\
&(u_0,u_2,\alpha_1,\ldots,\alpha_{k-2}),~
(u_3,u_2,\alpha_1,\ldots,\alpha_{k-2}),~
(u_3,u_0,\alpha_1,\ldots,\alpha_{k-2}),~
(u_3,u_1,\alpha_1,\ldots,\alpha_{k-2}),~\\
&(u_0,u_1,\alpha_1,\ldots,\alpha_{k-2}) 
\end{align*}
is a cycle in $F^{1}_{k\times 1}(G)$. 
Assume next that $\Delta(G)\le 2$. 
Then, $G$ must be a path $P_n$ or a cycle $C_n$. 
Moreover, $F^{1}_{k\times 1}(P_n)$ and $F^{1}_{k\times 1}(C_n)$ contain at least one cycle for $k\le n-2$ (the proofs are shown in Propositions \ref{prop:path-cycle} and \ref{prop:cycle-cycle}). 
It follows that $F^{1}_{k\times 1}(G)$ contains at least one cycle for $2\le k\le n-2$.
Together with (\ref{eq:subset2}), for any $1\le s\le k$, $F^{s}_{k\times 1}(G)$ contains at least one cycle.
\end{proof}

Denote by $d_G(u,v)$ the distance between $u$ and $v$ in $G$.
For a connected graph $G$, the following theorem gives a sufficient condition for all cases of supertoken graphs to be connected.
\begin{theorem}\label{th:connected-all}
    Let $G=(V,E)$ be a connected graph with maximum degree $\Delta(G)$. 
    For $1\le s\le k$, if $k<\Delta(G)$, then the supertoken graphs  $F_{k}^s(G)$ and $F^s_{k\times 1}(G)$ are connected. Moreover, for $\Delta(G) \le k\le n$, the supertoken graph $F_{k}^s(G)$ is connected, except for $F_{n}^1(G)$.
\end{theorem}
\begin{proof}
    We first show that $F_{k}^s(G)$ is connected for $1\le s\le k\le n$, 
    except for $F_{n}^1(G)$.  
    The cases $k=1$ and $(k,s)=(n,1)$ are trivial.   
    We prove it by induction on $s$ with $1\le s\le k$. 
    For $s=1$, $F_k^1(G)$ is connected, since $G$ is connected (see, for instance, Dalf\'o, Duque, Fabila-Monroy, Fiol, Huemer, Trujillo-Negrete, and
    Zaragoza Mart\'{\i}nez \cite{ddffhtz2021}, and Barik and Verma \cite{bv24}). 
    Suppose that $F_k^s(G)$ is connected for $1\le s\le k-1$. We next show that $F_k^{s+1}(G)$ is also connected. 
    Since the $k$ tokens are indistinguishable in $F_k^{s}(G)$, the vertices in $V(F_k^{s+1}(G))\backslash V(F_k^s(G))$ can be partitioned into $V_1,V_2,\ldots,V_{\lfloor\frac{k}{s+1}\rfloor}$ such that  
    \begin{eqnarray*}
        V_i:=\left\{A_i\in V(F_k^{s+1}(G))\backslash V(F_k^s(G)):A_i=(\underbrace{u_1,\ldots,u_1}_{s+1},\ldots,\underbrace{u_i,\ldots,u_i}_{s+1},u_{i+1},\ldots,u_r)\right\},
    \end{eqnarray*}
    where the number of times a vertex of $G$ appears with a token in 
    $u_{i+1},\ldots,u_r$ is less than $s+1$.
    Let $V_0=V(F^s_k(G))$. 
    It suffices to prove that for any vertex $A_i\in V_i$ with $i=1,\ldots,\lfloor\frac{k}{s+1}\rfloor$,  there exists a path from $A_i$ to one vertex $A_{i-1}$ in $V_{i-1}$. 
    Since $s+1\ge 2$ and $k\le n$, there exists a vertex $w\in V(G)\backslash \{u_1,\ldots,u_r\}$. Moreover, there exists a path from one of $u_1,\ldots,u_{i}$ to $w$ in $G$, as $G$ is connected. 
    Without loss of generality, assume that $u_i$ is the vertex among $\{u_1,\ldots,u_i\}$ satisfying $d_G(u_i,w)=\min_{1\le j\le i}\{d_G(u_j,w)\}$. 
    Denote the path from $u_i$ to $w$ by $P=u_iw_1w_2\cdots w_pw$. It follows that $w_1,\ldots,w_p\notin \{u_1,\ldots,u_{i-1}\}.$
    Let $A_{i-1}\in V_{i-1}$ such that $A_i\triangle A_{i-1}=\{u_i,w\}$ and $B_1,B_2,\ldots,B_p\in V_{i}\cup V_{i-1}$ such that $A_i\triangle B_j=\{u_i,w_j\}$ for $j=1,2,\ldots,p$.
    Then, there exists a path $A_{i}B_1B_2\cdots B_pA_{i-1}$ in $F_k^{s+1}(G)$, as desired.

    Next, consider $F^s_{k\times 1}(G)$ for $1\le s\le k<\Delta(G)$. 
    We use induction on $s$. For $s=1$, we will prove that $F^1_{k\times 1}(G)$ is connected in the following claim. 
    For $1\le s\le k-1$, using a similar approach as above, we find that $F^{s+1}_{k\times 1}(G)$ is connected if $F^{s}_{k\times 1}(G)$ is connected. 
    Let us go back to the case where $s=1$. 
    Let $v_0\in V$ be the vertex with maximum degree $\Delta=\Delta(G)$, and $v_1,\ldots,v_{\Delta}$ be the neighbors of $v_0$. 
    Let $V_S=\{v_0,v_1,\ldots,v_{\Delta}\}$. Let $F^1_{k\times 1}(V_S)$ be the subgraph of $F^1_{k\times 1}(G)$ induced by the vertices in $V_S$.

    {\bf Claim:} $F^1_{k\times 1}(V_S)$ is connected for $k<\Delta(G)$.
    
    Let $S_{\Delta}$ (with $\Delta+1$ vertices) be a star with vertex set $V_S$ and $v_0$ be its central vertex, and $F^1_{k\times 1}(S_{\Delta})$ be the supertoken graph of $S_{\Delta}$.
    Note that $F^{1}_{k\times 1}(S_{\Delta})$ is a subgraph of $F^1_{k\times 1}(V_S)$. 
    It suffices to prove that $F^1_{k\times 1}(S_{\Delta})$ is connected. 
    Let $A'_1,A'_2,\ldots,A'_p$ be the vertices of $F^{1}_k(S_{\Delta})$ with $p=\binom{\Delta+1}{k}$,
    and $\AA_i$ be the set of vertices, where each vertex in $\AA_i$ corresponds to a distinct permutation of the tokens in the vertex $A'_i$ for $i=1,\ldots,p$, which implies that $|\AA_i|=k!$ for every $i=1,\ldots,p$. 
    Note that $V(F^{1}_{k\times 1}(S_{\Delta}))=\AA_1\cup\AA_2\cup\cdots\cup\AA_p$. 
    Let $(\alpha_1,\alpha_2,\ldots,$ $\alpha_k)$ be one vertex of $F^{1}_{k\times 1}(S_{\Delta})$.
    We show that there is a path between every pair of vertices in $\AA_i$, where $i=1,\ldots,p$, that is, there is a path from $(\alpha_1,\ldots,\alpha_k)$ to $(\sigma(\alpha_1),\ldots,\sigma(\alpha_k))$, where $\sigma$ is a permutation of $\{\alpha_1,\ldots,\alpha_k\}$.   

    {\bf Case 1.} $v_0\in \{\alpha_1,\ldots,\alpha_k\}$.

    Without loss of generality, suppose that $\alpha_1=v_0$. 
    Since $k\le \Delta-1$, there are at least two vertices $u,v\in V_S\backslash \{v_0,\alpha_2,\ldots,\alpha_k\}$.
    Let $\sigma(v_0)=\alpha_{j_1}$ with $1\le j_1\le k$. 
    Note that there exists $j_2$ with $2\le j_2\le k$ and $j_2\ne j_1$ such that $\sigma(\alpha_2)=\alpha_{j_2}$, which implies $\sigma(\alpha_2)\ne v_0$. 
    We have 
    \begin{eqnarray}
    (v_0,\alpha_2,\ldots,\alpha_{j_2},\ldots,\alpha_k)&\sim&(u,\alpha_2,\ldots,\alpha_{j_2},\ldots,\alpha_k)\notag\\
        &\sim&(u,\alpha_2,\ldots,v_0,\ldots,\alpha_k)\notag\\
        &\sim&(u,\alpha_2,\ldots,v,\ldots,\alpha_k)        \notag\\
        &\sim&(u,v_0,\ldots,v,\ldots,\alpha_k)\notag\\
        &\sim&(u,\alpha_{j_2},\ldots,v,\ldots,\alpha_k)\notag\\
         &\sim&(v_0,\alpha_{j_2},\ldots,v,\ldots,\alpha_k)=(v_0,\sigma(\alpha_2),\ldots,v,\ldots,\alpha_k).\label{eq:connected-star1}
    \end{eqnarray}
    There are still at least two vertices $u,\alpha_{2}\in V_S\backslash \{v_0,\sigma(\alpha_2),\ldots,v,\ldots,\alpha_k\}$. 
    Then, by a similar analysis as above, we get that there is a path from $(v_0,\sigma(\alpha_2),\ldots,v,$
    $\ldots,\alpha_k)$ to $(v_0,\sigma(\alpha_2),$
    $\ldots,\sigma(\alpha_{j_1-1}),v^*,\sigma(\alpha_{j_1+1}),\ldots,\sigma(\alpha_k))$, where $v^*\ne \alpha_{j_1}$. 
    Moreover, we obtain 
    \begin{eqnarray}
    &&\left(v_0,\sigma(\alpha_2),\ldots,\sigma(\alpha_{j_1-1}),v^*,\sigma(\alpha_{j_1+1}),\ldots,\sigma(\alpha_k)\right)\notag\\
    &\sim&\left(\alpha_{j_1},\sigma(\alpha_2),\ldots,\sigma(\alpha_{j_1-1}),v^*,\sigma(\alpha_{j_1+1}),\ldots,\sigma(\alpha_k)\right)\notag\\
    &\sim&\left(\alpha_{j_1},\sigma(\alpha_2),\ldots,\sigma(\alpha_{j_1-1}),v_0,\sigma(\alpha_{j_1+1}),\ldots,\sigma(\alpha_k)\right)\notag\\
    &\sim&\left(\alpha_{j_1},\sigma(\alpha_2),\ldots,\sigma(\alpha_{j_1-1}),\sigma(\alpha_{j_1}),\sigma(\alpha_{j_1+1}),\ldots,\sigma(\alpha_k)\right)\notag\\
    &=&\left(\sigma(v_0),\sigma(\alpha_2),\ldots,\sigma(\alpha_{j_1-1}),\sigma(\alpha_{j_1}),\sigma(\alpha_{j_1+1}),\ldots,\sigma(\alpha_k)\right).\label{eq:connected-star2}
    \end{eqnarray}

    {\bf Case 2.} $v_0\notin \{\alpha_1,\ldots,\alpha_k\}$.

    Let $\sigma$ be a permutation of $\{\alpha_1,\ldots,\alpha_k\}$ with $\sigma(\alpha_1)=\alpha_{j_1}$ and $\sigma(\alpha_{j_2})=\alpha_1$, and $\tau$ be a permutation of $\{v_0,\alpha_2,\ldots,\alpha_k\}$ such that $\tau(v_0)=\alpha_{j_1}$, $\tau(\alpha_{j_2})=v_0$, and $\tau(\alpha_i)=\sigma(\alpha_i)$ for $2\le i\le k$ and $i\ne j_2$.
    It follows from Case 1 that there is a path from $(v_0,\alpha_2,\ldots,\alpha_k)$ to $(\tau(v_0),\tau(\alpha_2),\ldots,\tau(\alpha_k))=(\alpha_{j_1},\sigma(\alpha_2),\ldots,v_0,\ldots,\sigma(\alpha_k))$. 
    Note that 
    \begin{equation}\label{eq:connected-star3}
        (v_0,\alpha_2,\ldots,\alpha_k)\sim (\alpha_1,\alpha_2,\ldots,\alpha_k)
    \end{equation}
    and
    \begin{eqnarray}
        (\alpha_{j_1},\sigma(\alpha_2),\ldots,v_0,\ldots,\sigma(\alpha_k))&\sim&(\alpha_{j_1},\sigma(\alpha_2),\ldots,\alpha_1,\ldots,\sigma(\alpha_k))\notag\\
        &=&(\sigma(\alpha_1),\sigma(\alpha_2),\ldots,\sigma(\alpha_{j_2}),\ldots,\sigma(\alpha_k))\label{eq:connected-star4}.
    \end{eqnarray}
    Thus, there is a path between every pair of vertices in $\AA_i$ for $i=1,\ldots,p$.
    
    As shown before, we get that $F^{1}_k(S_{\Delta})$ is connected for $1\le s\le k$. 
    This implies that there is a path from $A'_i$ to $A'_j$ in $F^{1}_k(S_{\Delta})$ for every $i,j=1,\ldots,p$ and $i\ne j$.
    Thus, for any vertex $B'_i\in \AA_i$ and $B'_j\in \AA_j$ with $i,j=1,\ldots,p$ and $i\ne j$, we have that there 
    exists a path from $B'_i$ to $B'_j$ in $F^{1}_{k\times 1}(S_{\Delta})$, that is, it is connected. 
    This completes the proof of the claim. 

    Since $F^{1}_{k\times 1}(V_S)$ 
    is a subgraph of $F^{1}_{k\times 1}(G)$, it suffices to show that for any vertex $A^*\in V(F^{1}_{k\times 1}(G))\backslash V(F^{1}_{k\times 1}(V_S))$, there is a path from $A^*$ to one of the vertices in $F^{1}_{k\times 1}(V_S)$.
    Consider the vertex $(\alpha_1^*,\alpha^*_2,\ldots,\alpha^*_k)$ in $V(F^{1}_{k\times 1}(G))$.
    Let \begin{eqnarray*}
        V^*_0:=\left\{(\alpha^*_1,\alpha^*_2,\ldots,\alpha^*_k)\in V(F^{1}_{k\times 1}(G)):\alpha^*_1,\ldots,\alpha^*_k\in V_S\right\},
    \end{eqnarray*}    
    and
    \begin{eqnarray*}
        V^*_i:=\left\{(\alpha^*_1,\alpha^*_2,\ldots,\alpha^*_k)\in V(F^{1}_{k\times 1}(G)):
        \alpha^*_{j_1},\ldots,\alpha^*_{j_i}\in V\backslash  V_S~\mbox{and}~\alpha^*_{j_{i+1}},\ldots,\alpha^*_{j_k}\in V_S\right\},
    \end{eqnarray*}
    for every $i=1,2,\ldots,k$.
    Note that $V^*_0=V(F^{1}_{k\times 1}(V_S))$ and $V(F^{1}_{k\times 1}(G))\backslash V(F^{1}_{k\times 1}(V_S))$
    $=V^*_1\cup V^*_2\cup \cdots\cup V^*_k$. 
    Thus, it suffices to prove that, for any vertex $A^*_i\in V^*_i$ with $i=1,\ldots,k$, there exists a path from $A^*_i$ to a vertex $A^*_0$ in $V^*_0$. 
    We use induction on $i$, where $1\le i\le k$. 
    For $i=1$, without loss of generality, suppose that $A^*_1=(\alpha^*_1,\alpha^*_2,\ldots,\alpha^*_k)\in V^*_1$ such that $\alpha^*_1\in V\backslash V_S$ and $\alpha^*_2,\ldots,\alpha^*_k\in V_S$. 
    Then, there exists a path from $\alpha^*_1$ to a vertex of $\{\alpha^*_2,\ldots,\alpha^*_k\}$~(since $G$ is connected). 
    Without loss of generality, assume that 
    $\alpha^*_t$ is the vertex among $\{\alpha^*_2,\ldots,\alpha^*_k\}$ satisfying $d_G(\alpha^*_1,\alpha^*_t)=\min_{2\le j\le k}d_G(\alpha^*_1,\alpha^*_j)$. 
    Denote the path from $\alpha^*_1$ to $\alpha^*_t$ by $P^*=\alpha^*_1w^*_1w^*_2\cdots w^*_p\alpha^*_t$. 
    If there exists a vertex $w^*_j$ such that $w^*_j\in V_S$, then we have 
    \begin{eqnarray*}
        A^*_1&=&(\alpha^*_1,\alpha^*_2,\ldots,\alpha^*_k)\notag\\
        &\sim&(w^*_1,\alpha^*_2,\ldots,\alpha^*_k)~(\mbox{token}~1~\mbox{moved from}~\alpha^*_1~\mbox{to}~w^*_1)\notag\\
        &&\cdots\notag\\
        &\sim&(w^*_j,\alpha^*_2,\ldots,\alpha^*_k)=A^*_0.\label{eq:connectedgraph1}
    \end{eqnarray*}    
    Consider next that $w^*_1,\ldots,w^*_p\notin V_S$.
    It implies that $\alpha^*_t\ne v_0$.
    For $v_0\notin \{\alpha^*_2,\ldots,\alpha^*_k\}$, we obtain 
     \begin{eqnarray*}
        A^*_1&=&(\alpha^*_1,\ldots,\alpha^*_t,\ldots,\alpha^*_k)\notag\\
         &\sim&(\alpha^*_1,\ldots,v_0,\ldots,\alpha^*_k)~(\mbox{token}~t~\mbox{moved from}~\alpha^*_t~\mbox{to}~v_0)\notag\\
        &\sim&(w^*_1,\ldots,v_0,\ldots,\alpha^*_k)\notag\\
        &&\cdots\notag\\
        &\sim&(w^*_p,\ldots,v_0,\ldots,\alpha^*_k)\notag\\
        &\sim&(\alpha^*_t,\ldots,v_0,\ldots,\alpha^*_k)=A^*_0.\label{eq:connectedgraph2}
    \end{eqnarray*}
    For $v_0\in \{\alpha^*_2,\ldots,\alpha^*_k\}$, suppose that $\alpha^*_r=v_0$ with $r\ne t$. 
    Thus, 
    \begin{eqnarray*}
        A^*_1&=&(\alpha^*_1,\ldots,\alpha^*_t,\ldots,\alpha^*_r,\ldots,\alpha^*_k)\notag\\
         &\sim&(\alpha^*_1,\ldots,\alpha^*_t,\ldots,u,\ldots,\alpha^*_k)~(\mbox{token}~r~\mbox{moved from}~\alpha^*_r=v_0~\mbox{to}~u\in V_S\backslash\{\alpha^*_2,\ldots,\alpha^*_k\})\notag\\
        &\sim&(\alpha^*_1,\ldots,v_0,\ldots,u,\ldots,\alpha^*_k)~(\mbox{token}~t~\mbox{moved from}~\alpha_t~\mbox{to}~v_0)\notag\\
        &\sim&(w^*_1,\ldots,v_0,\ldots,u,\ldots,\alpha_k)\notag\\
        &&\cdots\notag\\
        &\sim&(w^*_p,\ldots,v_0,\ldots,u,\ldots,\alpha_k)\notag\\
        &\sim&(\alpha^*_t,\ldots,v_0,\ldots,u,\ldots,\alpha^*_k)=A^*_0.\label{eq:connectedgraph3}
    \end{eqnarray*}
    Suppose next that for any vertex $A^*_i\in V^*_i$, there exists a path from $A^*_i$ to a vertex $A^*_0$ in $V^*_0$, where $1\le i\le k-1$. 
    Without loss of generality, assume that $A^*_{i+1}=(\alpha^*_1,\ldots,\alpha^*_{i+1},\alpha^*_{i+2},\ldots,\alpha^*_k)\in V^*_{i+1}$ such that $\alpha^*_1,\ldots,\alpha^*_{i+1}\in V\backslash V_S$ and $\alpha^*_{i+2},\ldots,\alpha^*_k\in V_S$. 
    Using a similar approach as above, we find that there exists a path from a vertex $A^*_{i+1}$ to a vertex $A^*_i\in V^*_i$. 
    Together with the hypothesis, there is a path from $A^*_{i+1}$ to a vertex $A^*_0$ in $V^*_0$, as desired.
\end{proof}

\begin{example}
Consider the graph $G$ shown in Figure \ref{fig:example-connected} and its supertoken graphs $F^{s}_{k}(G)$ and $F^{s}_{k\times 1}(G)$ with $k=4<5=\Delta(G)$ and $s=2$.

Note that $F_4^1(G)$ is connected.
Let $V_0=V(F_4^1(G))$. 
We get that the vertices in $V(F_4^2(G))\backslash V(F_4^1(G))$ can be partitioned into $V_1,V_2$ such that 
\begin{equation*}
        V_1:= \left\{(u_1,u_2,u_3,u_4):u_{j_1}=u_{j_2},~\mbox{and}~u_{j_1},u_{j_3},u_{j_4}~\mbox{are distinct}\right\}
    \end{equation*}
    and
    \begin{equation*}
        V_2:= \left\{(u_1,u_2,
        u_3,u_4):u_{j_1}=u_{j_2}\ne u_{j_3}=u_{j_4}\right\}.
    \end{equation*}
    We show that there is a path from any vertex $A_i\in V_i$ to one vertex $A_{i-1}\in V_{i-1}$ in $F_4^2(G)$.
    First, for $i=1$, consider a vertex $A_1=(6,6,2,0)\in V_1$. 
    We see that $1\notin\{6,2,0\}$ and the path from $6$ to $1$ is $P=6201$. 
    Let $B_1=(6,2,2,0)$, 
    $B_2=(6,0,2,0)$ and $A_0=(6,1,2,0)\in V_0$, so that $A_1\triangle B_1=\{6,2\}$, $B_1\triangle B_2=\{2,0\}$, and $B_2\triangle A_0=\{0,1\}$. 
    Then, the path is $A_1B_1B_2A_0$.
    Next, for $i=2$, consider a vertex $A_2=(6,6,2,2)\in V_2$. 
    Recall that $0\sim2$. 
    Then, the path is $A_2A_1$.
    Hence, we obtain a path from a vertex $A_i\in V_i$ (for $i=1,2$) to a vertex $A_{i-1}\in V_0$.

Next, consider $F^{2}_{4\times 1}(G)$. 
Using a similar analysis as above, we just show that $F^{1}_{4\times 1}(G)$ is connected. 
Let $S_5$ be a star with vertex set $\{0,1,2,3,4,5\}$ and $0$ be its central vertex (we define $S_n$ as a graph with $n+1$ vertices). 
Recall that $F^{1}_{4\times 1}(S_5)$ is a subgraph of $F^{1}_{4\times 1}(G)$. 
We first show that $F^{1}_{4\times 1}(S_5)$ is connected.
In this case, let $A'_1,A'_2,\ldots,A'_{15}$ be the vertices of $F^{1}_4(S_5)$ (with indistinguishable tokens), and $\AA_i$ be the set of vertices such that each vertex in $\AA_i$ corresponds to a distinct permutation of the tokens in the vertex $A'_i$ for $i=1,\ldots,15$.
So, the vertices of $F^{1}_{4\times 1}(S_5)$ are $\bigcup_{i=1}^{15} \AA_i$.
Now, let $B'_1=(1,2,3,4)\in \AA_1$ be a vertex in $F^{1}_{4\times 1}(S_5)$.
We want to show there is a path from vertex $B'_1$ to vertex $C'_1$ in $\AA_1$.
    Suppose that $\sigma$ is a permutation of $\{1,2,3,4\}$ such that $\sigma(1)=3,\sigma(2)=1,\sigma(3)=4,\sigma(4)=2$, that is, $C'_1=(\sigma(1),\sigma(2),\sigma(3),\sigma(4))=(3,1,4,2)$.
    Let $\tau$ be a permutation of $\{0,2,3,4\}$ such that $\tau(0)=3,\tau(2)=0,\tau(3)=\sigma(3)=4,\tau(4)=\sigma(4)=2$. 
    As shown in (\ref{eq:connected-star1}), we have 
    \begin{equation*}\label{eq:star-exmp1}
    (0,2,3,4)\sim(5,2,3,4)\sim(5,2,3,0)\sim(5,2,3,1)\sim(5,2,0,1)\sim(5,2,4,1)\sim(0,2,4,1).
    \end{equation*}
    Similarly, we get 
    \begin{equation*}\label{eq:star-exmp2}
        (0,2,4,1)\sim(3,2,4,1)\sim(3,0,4,1)\sim(3,5,4,1)\sim (3,5,4,0)\sim(3,5,4,2)\sim(0,5,4,2).
    \end{equation*}
    As shown in (\ref{eq:connected-star2}), we obtain
    \begin{equation*}\label{eq:star-exmp3}
        (0,5,4,2)\sim(3,5,4,2)\sim(3,0,4,2).
    \end{equation*}
    Then, there exists a path from $(0,2,3,4)$ to $(\tau(0),\tau(2),\tau(3),$
    $\tau(4))=(3,0,4,2)$. 
    Moreover, from (\ref{eq:connected-star3}) and (\ref{eq:connected-star4}), we have $(0,2,3,4)\sim (1,2,3,4)$ and $(3,0,4,2)\sim(3,1,4,2)$. 
    Thus, there exists a path from $B'_1=(1,2,3,4)$ to $C'_1=(3,1,4,2)$.

    Then, we show that, for any vertex $A^*\in V(F^{1}_{4\times 1}(G))\backslash V(F^{1}_{4\times 1}(S_5))$, there is a path from $A^*$ to one of the vertices in $V(F^{1}_{4\times 1}(S_5))$.
    In this case, let
     \begin{eqnarray*}
        V^*_0:=\left\{(\alpha_1,\alpha_2,\alpha_3,\alpha_4)\in V(F^{1}_{4\times 1}(G)):\alpha_1,\alpha_2,\alpha_3,\alpha_4\in \{0,1,2,3,4,5\}\right\},
    \end{eqnarray*}    
    and
    \begin{align*}
        V^*_i:=\Big\{(\alpha_1,\alpha_2,\alpha_3,\alpha_4)\in V(F^{1}_{4\times 1}(G)):
        &\alpha_{j_1},\ldots,\alpha_{j_i}\in \{6,7,8,9\}~\mbox{and}\\ &\alpha_{j_{i+1}},\ldots,\alpha_{j_4}\in \{0,1,2,3,4,5\}\Big\},
    \end{align*}
    for every $i=1,2,3,4$.
    We want to show that there is a path from vertex $A^*_i\in V^*_i$ to a vertex in $V^*_0$.
    First, let $A^*_1=(7,1,0,4)\in V^*_1$ with $d_G(7,0)<d_G(7,1)=d_G(7,4)$ and $P^*_1=7620$. 
    Observe that $2\in V(S_5)$.
    It follows that 
    \begin{equation*}
        A^*_1=(7,1,0,4)\sim(6,1,0,4)\sim(2,1,0,4)\in V^*_0.
    \end{equation*}
    Suppose that $A^*_2=(7,6,0,4)\in V^*_2$.
    Note that $d_G(6,0)<d_G(7,0)$. 
    It follows that 
    \begin{equation}\label{eq:a2}
        A^*_2=(7,6,0,4)\sim(7,2,0,4)\sim(7,2,3,4)\sim(7,0,3,4)\sim(6,0,3,4)\sim(2,0,3,4)\in V^*_0.
    \end{equation}
    Next, let $A^*_3=(7,6,8,4)\in V_3$.
    Recall that $6\sim 2$ and $8\sim 3$. 
    Together with (\ref{eq:a2}), we have 
    \begin{equation*}
        A^*_3=(7,6,8,4)\sim(7,6,3,4)\sim(7,2,3,4)\sim(7,0,3,4)\sim(6,0,3,4)\sim(2,0,3,4)\in V^*_0.
    \end{equation*}
    Now, assume that $A^*_4=(7,6,8,9)$.
    It follows that 
    \begin{eqnarray*}
       A^*_4&=&(7,6,8,9)\sim(7,6,3,9)\sim (7,6,0,9)\sim(7,6,0,8)\sim(7,6,0,3)\sim(7,6,1,3)\\
       &\sim&(7,2,1,3)\sim(7,0,1,3)\sim(6,0,1,3)\sim(2,0,1,3)\in V^*_0. 
    \end{eqnarray*}
    So, we obtain a path from a vertex $A^*_i\in V^*_i$ (for $i=1, 2, 3,4$) to a vertex $A^*_0\in V_0=V(F^{1}_{4\times 1}(S_5))$.

    \begin{figure}   
    \begin{center}
    \begin{tikzpicture}[scale=1.1,auto,swap]
    \vertex (0) at (0,0) [label=above:{{\footnotesize$0$}}]{};
    \vertex (2) at (1,0.6) [label=above:{{\footnotesize $2$}}]{};
    \vertex (5) at (1,-1.2) [label=above:{{\footnotesize $5$}}]{};
    \vertex (9) at (2,0.6) [label=above:{{\footnotesize $6$}}]{};
    \vertex (6) at (3,0.6) [label=above:{{\footnotesize $7$}}]{};
    \vertex (3) at (1,0) [label=above:{{\footnotesize $3$}}]{};
    \vertex (7) at (2,0) [label=above:{{\footnotesize $8$}}]{};
    \vertex (8) at (3,0) [label=above:{{\footnotesize $9$}}]{};
    \vertex (1) at (1,1.2) [label=above:{{\footnotesize $1$}}]{};
    \vertex (4) at (1,-0.6) [label=above:{{\footnotesize $4$}}]{};
    \draw [line width=0.6pt] (0)--(1);
    \draw [line width=0.6pt] (0)--(2);
    \draw [line width=0.6pt] (0)--(5);
    \draw [line width=0.6pt] (0)--(3);
    \draw [line width=0.6pt] (0)--(4);
    \draw [line width=0.6pt] (2)--(9);
    \draw [line width=0.6pt] (6)--(9);
    \draw [line width=0.6pt] (3)--(7);
    \draw [line width=0.6pt] (8)--(7);
    \end{tikzpicture}
    \caption{A connected graph $G$ with maximum degree $\Delta(G)=5$. 
    }
     \label{fig:example-connected}
    \end{center}
\end{figure}
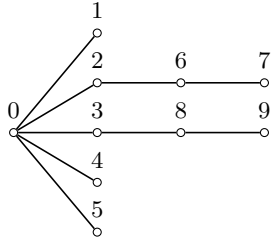
\end{example}

\begin{proposition}\label{prop:bipartite-star}
    For $1\le s\le k \le n$, the supertoken graph $F_k^s(S_n)$ and $F^s_{k\times 1}(S_n)$ of a star $S_n$ (with $n+1$ vertices) are bipartite.
\end{proposition}
\begin{proof}
    Let $V(S_n)=\{0,1,\ldots,n\}$.
    Note that the vertices in $V(F_k^s(S_n))$ (or $V(F^s_{k\times 1}(S_n))$) can be partitioned into $V_0,V_1,\ldots,V_{\lfloor\frac{k}{s}\rfloor}$ such that $$V_i:=\left\{A=\right(\alpha_1,\alpha_2,\ldots,\alpha_k):\alpha_{j_1}=\cdots=\alpha_{j_i}=0~\mbox{for}~j_1,\ldots,j_{i}\in [k]\}.$$
    Let $X$ and $Y$ be the unions of all $V_i$ with odd and even indices $i$, respectively. 
    Then no edges join pairs of vertices belonging to $X$, and the same is true in $Y$.
\end{proof}

\section{The token graphs $F_k^1(G)=F_k(G)$}
\label{sec:token}
In this case, the $k$ tokens are indistinguishable, and the maximum number of tokens per vertex is one.
The token graph $F^{1}_{k}(G)=F_k^1(G)$
has order $\binom{n}{k}$ and size $\binom{n-2}{k-1} m$. 
$F_k(G)$ is the known {\em $k$-th symmetric power} of $G$ (see Audenaert, Godsil, Royle, and Rudolph~\cite{agrr07}), later renamed {\em $k$-token graph} of $G$ by Fabila-Monroy, Flores-Pe\~{n}aloza, Huemer, Hurtado, Urrutia, and Wood~\cite{ffhhuw12}. 
See an example in Figure~\ref{C_4+2-supertokens(C_4)}$(b)$ for the case $G=C_4$ ($n=4$) and $k=2$, where $F_2^1(C_4)\cong K_{2,4}$.

In particular, when $G$ is the complete graph $K_n$, the token graph $F_k^1(K_n)$ is the distance-regular graph known as the {\em Johnson graph} $J(n,k)$, which is closely related to some issues of coding theory.
See, for instance, Godsil \cite{Go93}.
\subsection{Isomorphism between the token and supertoken graphs of paths}

When $G=P_n$, the path graph on $n$ vertices, we have the following result.

\begin{lemma}
For every $k<n$, the following isomorphism holds
$$
F_k^1(P_n)\cong F_{n-k}^{n-k}(P_{k+1}).
$$
\end{lemma}

\begin{proof}
Let us show a one-to-one mapping between the corresponding vertex sets that is a graph isomorphism.
Every vector $(\alpha_1,\ldots,\alpha_{k+1})$ representing a vertex of $F_{n-k}^{n-k}(P_{k+1})$ is mapped to a vertex of $F_k^1(P_n)$ in the following way:
\begin{itemize}
\item[$\circ$]
If $i\neq k+1$, then $\alpha_i$ is replaced by $\alpha_i$ 0's and one 1.
\item[$\circ$]
If $i= k+1$, then $\alpha_{k+1}$ is replaced by $\alpha_i$ 0's.
\end{itemize}
For instance, when $k+1=n-k=6$ ($n=11$), we get the following maps (the maps are defined from the vertices of $F_{6}^{6}(P_{6})$ to the vertices of $F_{5}^1(P_{11})$):
\begin{align*}
(3,0,0,1,0,2) & \quad\mapsto\quad (0,0,0,1,1,1,0,1,1,0,0) \ = \ (4,5,6,8,9);\\
(1,1,1,1,1,1) & \quad\mapsto\quad (0,1,0,1,0,1,0,1,0,1,0) \ = \ (2,4,6,8,10).
\end{align*}
This gives a vector $(\beta_1,\ldots,\beta_n)$ with $n-k$ 0's and $k$ 1's,
which corresponds to a vertex of $\alpha_i$ 0's. 
(As shown in the example, the positions of the 1's indicate the vertices of $P_n$ having a token). 
Then, in terms of the $\alpha_i's$, a vertex of $F_k^1(P_n)$ is a vector of $k$ components $(\beta'_1,\ldots,\beta'_k)$ (representing the vertices of $P_n$ having a token) computed as
\begin{align*}
\beta'_1 &=\alpha_1+1,\\
\beta'_2 &=\alpha_1+\alpha_2+2,\\
 & \vdots \\
\beta'_k &=\sum_{i=1}^k \alpha_i+k.
\end{align*}
Conversely, in terms of the $(\beta'_i)'s$, the vector of $k+1$ components representing a vertex of $F_{n-k}^{n-k}(P_{k+1})$ is
\begin{align*}
\alpha_1 &=\beta'_1-1,\\
\alpha_2 &=\beta'_2-\beta'_1-1,\\
 & \vdots \\
\alpha_k &=\beta'_k-\beta'_{k-1}-1,\\
\alpha_{k+1} &=n-\beta'_k.
\end{align*}
Moreover, in both graphs, the adjacencies correspond to moving one unit one step backward or forward, so that the mappings are an isomorphism, as claimed.
\end{proof}

For instance, Figure~\ref{fig:F_2^1(P_7)+F_5^5(P_3)} shows the isomorphic graphs $F_2^1(P_7)$ and $F_5^5(P_3)$. For example, notice the vertex equivalences
\begin{align*}
(1,3,1) & \quad\mapsto\quad (0,1,0,0,0,1,0) =(2,6),\\
(0,5,0) & \quad\mapsto\quad (1,0,0,0,0,0,1) =(1,7).
\end{align*}

\begin{figure}[t]
\begin{center}
\includegraphics[width=14cm]{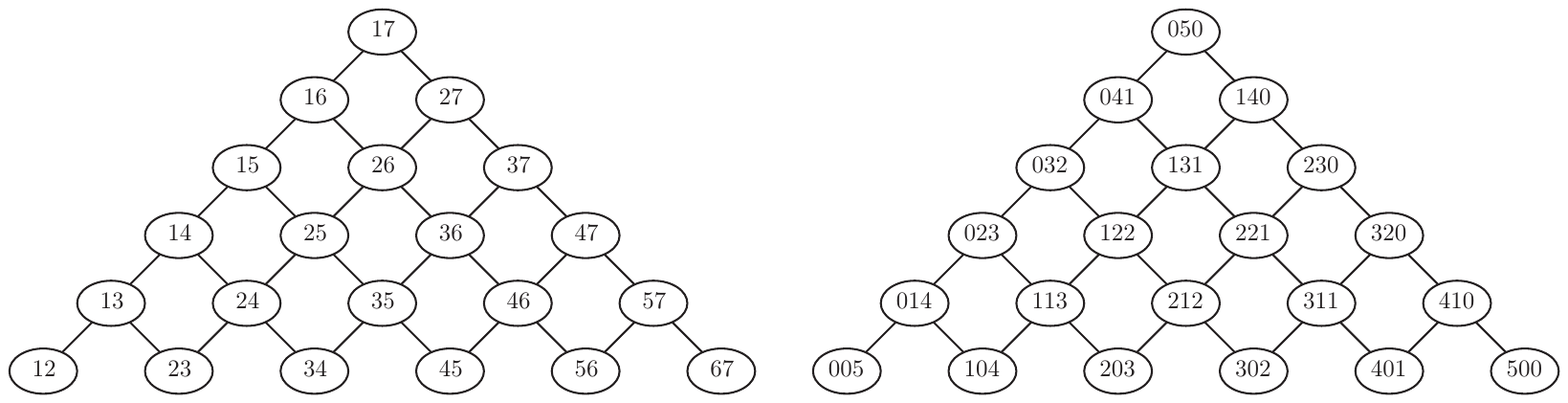}
\end{center}
\vskip-7.5cm
\caption{The graphs $F_2^1(P_7)$ (left) and $F_5^5(P_3)$ (right).}
\label{fig:F_2^1(P_7)+F_5^5(P_3)}
\end{figure}

More results on token graphs can be found in papers by Lea\~{n}os and Ndjatchi \cite{LN2021}, 
Lea\~{n}os and Trujillo-Negrete \cite{LT2018}, Dalf\'o, Duque, Fabila-Monroy, Fiol, Huemer,  Trujillo-Negrete, and Zaragoza Mart\'inez \cite{ddffhtz2021}, and Barik and Verma \cite{bv24}.

\section{The supertoken graphs $F_k^s(G)$}
\label{sec:F_k^s(G)}

In this case, the $k$ tokens are indistinguishable, and the maximum number of tokens per vertex is $s$, with $1<s< k$.
The supertoken graph $F^{s}_{k}(G)$ has order given by
the number of integer solutions $(x_1,\ldots,x_n)$, with $0\leq x_i \leq s$ for $i=1,\ldots,n$, of the equation $x_1+\cdots+x_n=k$. This
corresponds to the difference between the number of solutions with $x_i\geq0$ and the ones with $x_i\geq s+1$ for at least one $i$ from $\{1,\ldots,n\}$. 

\begin{proposition}
Let $G=(V,E)$ be a graph on $n$ vertices. The order and size of the supertoken graph $F_k^s(G)$ are, respectively,
\begin{equation}\label{nomial}
|V(F_k^s(G))| = \sum_{i=0}^{\lfloor k/(s+1)\rfloor} (-1)^i \binom{n}{i} \binom{n+k-1-i(s+1)}{n-1} := f(n,k,s),
\end{equation}
and 
\begin{equation}\label{eq:edge-fks}
|E(F_k^s(G))| = |E| \left(\sum_{i=1}^s i f(n-2,k-i,s) + 
\sum_{j=1}^{\omega-s+1} (s-j)f(n-2,k-s-j,s) \right),
\end{equation}
with $\omega=\min\{2s-2,k-1\}$. 
In (\ref{eq:edge-fks}), we define $f(n,0,s)=1$ when $j=k-s$.
\end{proposition}

\begin{proof}
    Let $S$ be the set of non-negative integer solutions $(x_1,\ldots,x_n)$ of the equation $x_1+\cdots+x_n=k.$
    Then, $|S|=|CR^n_k|=\binom{n+k-1}{k}$. 
    Among $x_1,\ldots,x_n$, at most $l=\lfloor k/(s+1)\rfloor$ numbers have values greater than $s$. 
    For $i=1,\ldots,l$, let $P_i$ be the set of non-negative integer solutions $(x_1,\ldots,x_n)$ with $x_i\ge s+1$. Then, $$|P_i|=\binom{n+k-(s+1)-1}{k-(s+1)}=\binom{n+k-(s+1)-1}{n-1}.$$
    Moreover, we have $$|P_{i_1}\cap \cdots\cap P_{i_j}|=\binom{n+k-j(s+1)-1}{k-j(s+1)}=\binom{n+k-j(s+1)-1}{n-1}$$ for $j=1,\ldots,l.$
    Note that the order of the supertoken graph $F_k^s(G)$ is exactly the difference between the number of solutions with $x_i\ge 0$ and the ones with $x_i\ge s+1$ for at least one $i$, with $i=1,\ldots,n$.
    By the inclusion-exclusion principle, we find that
    \begin{eqnarray*}
        |V(F_k^s(G))|&=&|S|-|P_{i_1}\cup P_{i_2} \cup \cdots \cup P_{i_l}|\\
        &=&|S|-\sum|P_i|+\sum|P_{i_1}\cap P_{i_2}|+\cdots+(-1)^l\sum|P_{i_1}\cap P_{i_2} \cap \cdots \cap P_{i_l}|\\
        &=&\sum_{i=0}^{\lfloor k/(s+1)\rfloor} (-1)^i \binom{n}{i} \binom{n+k-1-i(s+1)}{n-1}.
    \end{eqnarray*}
    
    Consider one edge $uv\in E$. Then, the edges corresponding to $uv$ in $F_k^s(G)$ are 
    \begin{equation*}
        \left\{A_uA_v:A_u=(u,\alpha_2,\ldots,\alpha_k), A_v=(v,\alpha_2,\ldots,\alpha_k)~\mbox{with}~\alpha_i\in V~\mbox{for}~i=2,\ldots,k\right\}.
    \end{equation*}
    Suppose that $t$ is the number of elements equal to $u$ or $v$ in $\{\alpha_2,\ldots,\alpha_k\}$.
    Let $\omega=\min\{2s-2,k-1\}$.
    It follows that $0\le t\le \omega$ and the remaining elements of $\alpha_2,\ldots,\alpha_k$ except for $\{u,v\}$ are chosen from $V\backslash\{u,v\}$, and there are $f(n-2,k-t-1,s)$ possible combinations with repetition. 
    Let $t_u$ and $t_v$ be the numbers of elements equal to $u$ and $v$, respectively, with $t_u+t_v=t$.
    First, consider $0\le t\le s-1$.
    The solutions on $t_u+t_v=t$ are 
    \begin{equation*}
        t_u=0,t_v=t;\ t_u=1,t_v=t-1;\ \cdots ;\ t_u=t,t_v=0. 
    \end{equation*}
    Then, there are $t+1$ possible combinations with repetition.
    Hence, the number of possible combinations with repetition of $\{\alpha_2,\ldots,\alpha_k\}$ with $0\le t\le s-1$ is 
    \begin{equation*}
    \sum_{t=0}^{s-1}(t+1)f(n-2,k-t-1,s)=\sum_{i=1}^s i f(n-2,k-i,s).
    \end{equation*}
    Next, consider $s\le t\le \omega$. 
    The solutions on $t_u+t_v=t$ are 
    \begin{equation*}
        t_u=s-1,t_v=t-(s-1);\ t_u=s-2,t_v=t-(s-1)+1;\ \cdots; \ t_u=t-(s-1),t_v=s-1.
    \end{equation*}
    Then, there are $2s-1-t$ possible combinations with repetition.
    It follows that the number of possible combinations with repetition of $\{\alpha_2,\ldots,\alpha_k\}$ with $s\le t\le \omega$ is 
    \begin{equation}\label{eq:valuej}
    \sum_{t=s}^{\omega}(2s-1-t)f(n-2,k-t-1,s)=\sum_{j=1}^{\omega-s+1} (s-j)f(n-2,k-s-j,s).
    \end{equation}
    Note that $j\le \omega-s+1$ and $\omega\le k-1$.
    The variable $j$ in equation (\ref{eq:valuej}) satisfies $j\le k-s$.
    Finally, we get the number of possible combinations with repetition of ${\alpha_2,\ldots,\alpha_k}$ is 
    \begin{equation*}
    \sum_{i=1}^s i f(n-2,k-i,s) + \sum_{j=1}^{\omega-s+1} (s-j)f(n-2,k-s-j,s),
    \end{equation*}
    where $\omega=\min\{2s-2,k-1\}$.
    Therefore, the equation (\ref{eq:edge-fks}) holds since $G$ has size $|E|$.
    \end{proof}

    Note that since we cannot have more than $k$ tokens in a vertex, any situation where $s>k$ is treated the same as $s=k$.
    
     As an example (see Figure \ref{fig:F_3^12(C_4)}), the order of $F_3^2(C_4)$ 
     is $f(4,3,2)=16$ and the size of $F_3^2(C_4)$ according to \eqref{eq:edge-fks} is
     $$
     E(F^2_3(C_4))=4\left(\sum_{i=1}^2 i\,f(2,2,2)+f(2,0,2)\right)=32.
     $$

     As another example, $F_4^3(C_4)$ has 31 vertices and 72 edges, coinciding with formulas (8) and (9). 
From (8), the order of $F_4^3(C_4)$ is 
\begin{equation*}
    f(4,4,3)=\binom{4+4-1}{4-1}-\binom{4}{1}\binom{4+4-1-(3+1)}{4-1}=\binom{7}{3}-4\binom{3}{3}=35-4=31.
\end{equation*}
Note that $\omega=\min\{2s-2,k-1\}=3$.
From (9), the size of $F_4^3(C_4)$ is
\begin{eqnarray*}
    &&4\left(\sum_{i=1}^3 i\,f(2,4-i,3)+\sum_{j=1}^1(3-j)f(2,4-3-j,3)\right)\\
    &=&4\left(1f(2,3,3)+2f(2,2,3)+3f(2,1,3)+2f(2,0,3)\right)\\
    &=&4\left(1\binom{4}{1}+2\binom{3}{1}+3\binom{2}{1}+2\right)\\
    &=&4\cdot 18=72.
\end{eqnarray*}
The vertices of $F_4^3(C_4)$ and the edges in $F_4^3(C_4)$ corresponding to the edge $12$ in $C_4$ are listed in Tables \ref{table:F_4^3(C_4)vertices} and \ref{table:F_4^3(C_4)edges}. 

    \begin{table}[t]
\begin{center}
\begin{tabular}{|ccccccccc|}
  \hline
  $(1234)$&$(1123)$   &$(1124)$&$(1134)$&$(2213)$&$(2214)$&$(2234)$&$(3312)$&$(3314)$ \\
  $(3324)$&$(4412)$   &$(4413)$&$(4423)$&$(1122)$&$(1133)$&$(1144)$&$(2233)$&$(2244)$\\
  $(3344)$&$(1112)$   &$(1113)$&$(1114)$&$(2221)$&$(2223)$&$(2224)$&$(3331)$&$(3332)$ \\
  $(3334)$&$(4441)$   &$(4442)$&$(4443)$& & & & &\\
  \hline
\end{tabular}
\caption{Vertices in $F_4^3(C_4)$.}

\label{table:F_4^3(C_4)vertices}
\end{center}
\end{table}
    \begin{table}[t]
\begin{center}
\begin{tabular}{|c|c|c|c|c|c|}
  \hline
  $(1334)$&$(1344)$  &$(1333)$&$(1244)$&$(1112)$&$(1122)$\\
  $(2334)$&$(2344)$   &$(2333)$&$(2444)$&$(2112)$&$(2122)$\\
  \hline
$(1134)$&$(1133)$&$(1144)$&$(1234)$&$(1233)$&$(1244)$ \\
  $(2134)$&$(2133)$&$(2144)$&$(2234)$&$(2233)$&$(2244)$\\
  \hline
  $(1113)$&$(1123)$   &$(1223)$&$(1114)$&$(1134)$&$(1224)$ \\
  $(2113)$&$(2123)$   &$(2223)$&$(2114)$&$(2134)$&$(2224)$\\
  \hline
\end{tabular}
\caption{Edges in $F_4^3(C_4)$ corresponding to the edge $12$ in $C_4$.}
\label{table:F_4^3(C_4)edges}
\end{center}
\end{table}

\begin{figure}[t]
\begin{center}
\includegraphics[width=6cm]{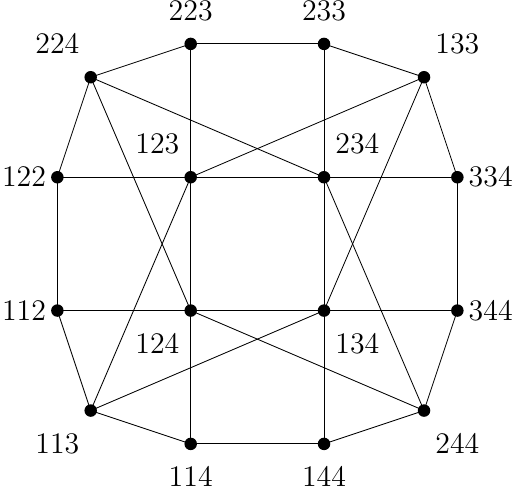}
\end{center}
\caption{The graph $F_3^2(C_4)$.}
\label{fig:F_3^12(C_4)}
\end{figure}
    
Alternatively, we note that, for $s\ge 0$, $f(n,k,s)$ in \eqref{nomial} is the $(s+1)$-nomial coefficient satisfying the recurrence
    $$
    f(n,k,s)=\sum_{i=0}^s f(n-1,k-i,s).
    $$ 
    For example, for $s=2$
    and $n=0,1,2,\ldots$, the trinomial coefficients turn out to be
    \begin{align*}
     &1\\
     &1,  1\\
     &1,2,3,2,1\\
     &1,3,6,7,6,3,1\\
     &1,4,10,16,19,16,10,4,1\\
     & \ldots
     \end{align*}

Let $G$ have vertices indexed by the integers $1,2,\ldots,n$.
Then, each vertex of the supertoken graph $F_{k}^s(G)$ 
can be represented by a vector $(\alpha_1,\ldots,\alpha_n)$, where $\alpha_i\in [0,s]$ is the number of tokens of the only color 
placed at the vertex $i\in [1,n]$.
In particular, the (classic) $k$-token graph $F_k(G)=F_k^1(G)$, with $0<k<n$, has vertices labeled by the binary $n$-vectors with $k$ 1's and $n-k$ 0's.
Then, the known isomorphism $F_k(G)\cong F_{n-k}(G)$ (see Fabila-Monroy, Flores-Pe\~{n}aloza, Huemer, Hurtado, Urrutia, and Wood~\cite{ffhhuw12}) is a consequence of the map $(\alpha_1,\ldots,\alpha_n)\mapsto (\overline{\alpha_1},\ldots,\overline{\alpha_n})$, where $\overline{0}=1$ and $\overline{1}=0$.

\section{The supertoken graphs $F_k^k(G)$}
\label{sec:F_k^k(G)}

Let $G$ be a graph of order $n$ and size $m$. 
In this case, the $k$ tokens are indistinguishable, and the maximum number of tokens per vertex is $k$.
The supertoken graph $F_k^k(G)$ has order $|V(F_k^k(G))|=CR^n_{k}=\binom{n+k-1}{k}$. Note that we can also get this value from \eqref{nomial} since this formula is also valid for $s=k$.

\begin{proposition}
  The supertoken graph $F_k^k(G)$ has size 
$$|E(F_k^k(G))|=m{\binom{n+k-2}{k-1}}.
$$
\end{proposition}
\begin{proof}
    Consider one edge $uv\in E(G)$. 
    Then, the edges corresponding to $uv$ in $F_k^k(G)$ are 
    \begin{equation*}
        \left\{A_uA_v:A_u=(u,\alpha_2,\ldots,\alpha_k), A_v=(v,\alpha_2,\ldots,\alpha_k)~\mbox{with}~\alpha_i\in V(G)~\mbox{for}~i=2,\ldots,k\right\}.
    \end{equation*}
    Since the maximum number of tokens per vertex is $k$ in $F_k^k(G)$, each vertex in $\{\alpha_2,\ldots,\alpha_k\}$ has at most $k-1$ tokens.
    Hence, the number of possible combinations with repetition of $\{\alpha_2,\ldots,\alpha_k\}$ is $\binom{n+k-2}{k-1}$. 
    The result holds as $G$ has size $m$.
\end{proof}

See an example in Figure~\ref{C_4+2-supertokens(C_4)}$(c)$ for the case $G=C_4$ ($n=4$) and $k=2$. This kind of supertoken graph was introduced by Hammack and Smith~\cite{HaSm17}, who named them \emph{reduced power graphs}.


If $G=K_n$,
then the vertices of the supertoken graph $F_k^k(K_n)$ represent the different possible states of a multiprocessor with $n$ memory modules and $k$ (indistinguishable) processors.

\section{The supertoken graphs $F^{1}_{k\times 1}(G)$}
\label{sec:k-colors-1tokenvertex}

In this case, the $k$ tokens are different, and the maximum number of tokens per vertex is one.
The supertoken graph $F^{1}_{k\times 1}(G)$ has order $\frac{n!}{(n-k)!}$ and size $mk\frac{(n-2)!}{(n-k-1)!}$.
See an example in Figure~\ref{C_4+2-supertokens(C_4)}$(d)$ for the case $G=C_4$ ($n=4$) and 2 tokens of different colors.

\subsection{The supertoken graphs of paths and cycles}

Let $G$ be a graph with $V(G)=\{1,2,\ldots,n\}$. 
For a vertex $(\alpha_1,\alpha_2,\ldots,\alpha_k)$ in $F^{1}_{k\times 1}(G)$, where $\alpha_i\in [n]$ for $i=1,2,\ldots,k$, 
we define 
\begin{equation}\label{eq:decrease}
    (\alpha'_1,\alpha'_2,\ldots,\alpha'_k) \preceq (\alpha_1,\alpha_2,\ldots,\alpha_k)
\end{equation}
if there exists an integer $j$ such that $\alpha'_j<\alpha_j$, and $\alpha'_l=\alpha_l$ for $l\in [k]\backslash \{j\}$, and
\begin{equation}\label{eq:increase}
(\alpha''_1,\alpha''_2,\ldots,\alpha''_k)\succeq(\alpha_1,\alpha_2,\ldots,\alpha_k)
\end{equation}
if there exists an integer $j$ such that $\alpha''_j>\alpha_j$, and $\alpha''_l=\alpha_l$ for $l\in [k]\backslash \{j\}$.
For a vertex $(\alpha_1,\alpha_2,\ldots,\alpha_k)$, 
we refer to the vertex $(\alpha'_1,\alpha'_2,\ldots,\alpha'_k)$ satisfying (\ref{eq:decrease}) as being on the left of $(\alpha_1,\alpha_2,\ldots,\alpha_k)$, and the vertex $(\alpha''_1,\alpha''_2,\ldots,\alpha''_k)$ satisfying (\ref{eq:increase}) as being on the right of $(\alpha_1,\alpha_2,\ldots,\alpha_k)$.\\

\begin{proposition}\label{prop:path-cycle}
    Let $G=P_n$ be a path with order $n$. Then, $F^{1}_{k\times 1}(G)$ contains at least one cycle for $k\le n-2$, and does not contain any cycle for $k=n-1,n$. Moreover, for $k=n$, $F^{1}_{k\times 1}(G)$ is the graph with $n!$ isolated vertices, and $F^{1}_{k\times 1}(G)=\underbrace{P_n \cup \cdots\cup P_n}_{(n-1)!}$ for $k=n-1$.
\end{proposition}
\begin{proof}
Define $V(P_n)=\{1,2,\ldots,n\}$. Let $(\alpha_1,\alpha_2,\ldots,\alpha_k)$ be one vertex of $F^{1}_{k\times 1}(P_n)$, where $\alpha_i\in [n]$ and $\alpha_1,\ldots,\alpha_k$ are distinct. The case $k=n$ is trivial. If $k\le n-2$, then we can choose $\alpha_3,\ldots,\alpha_k\in\{5,\ldots,n\}$, that is, there are tokens at vertices 1 and 3 ($\alpha_1=1$ and $\alpha_2=3$). Thus, $(1,3,\alpha_3,\ldots,\alpha_k)$, $(1,4,\alpha_3,\ldots,\alpha_k)$, $(2,4,\alpha_3,\ldots,\alpha_k)$, $(2,3,\alpha_3,\ldots,\alpha_k)$, and $(1,3,\alpha_3,\ldots,\alpha_k)$ is a cycle, as desired.
Suppose next that $k=n-1$. 
Consider the vertex $(\alpha_1,\alpha_2,\ldots,\alpha_k)$, where $1,n\in \{\alpha_1,\ldots,\alpha_k\}$.
Then, there exists exactly one vertex $(\alpha'_1,\alpha'_2,\ldots,\alpha'_k)$ satisfying $(\alpha'_1,\alpha'_2,\ldots,\alpha'_k) \preceq(\alpha_1,\alpha_2,\ldots,\alpha_k)$ on its left and exactly one vertex $(\alpha''_1,\alpha''_2,\ldots,\alpha''_k)$ satisfying $(\alpha''_1,\alpha''_2,\ldots,\alpha''_k)$ $\succeq(\alpha_1,\alpha_2,\ldots,\alpha_k)$ on its right.
Thus, starting from the vertex $(\alpha_1,\alpha_2,\ldots,\alpha_k)$, the leftmost vertex must be $(\tau(1),\tau(2),\ldots,\tau(n-1))$ and the rightmost vertex must be $(\sigma(2),\sigma(3),\ldots,$
$\sigma(n))$, where $\tau$ is a permutation on $\{1,\ldots,n-1\}$ and $\sigma$ is a permutation on $\{2,\ldots,n\}$. 
Moreover, we see that there are no more vertices adjacent to any of the vertices in the path starting from $(\alpha_1,\alpha_2,\ldots,\alpha_k)$ with end vertices $(\tau(1),\tau(2),\ldots,\tau(n-1))$ and $(\sigma(2),\sigma(3),\ldots,\sigma(n))$. 
A 
vertex $(\alpha^*_1,\alpha^*_2,\ldots,\alpha^*_k)$ with $1\notin \{\alpha^*_1,\alpha^*_2,\ldots,\alpha^*_k\}$ or $n\notin \{\alpha^*_1,\alpha^*_2,\ldots,\alpha^*_k\}$ must be included in the path starting from one vertex $(\alpha_1,\alpha_2,\ldots,\alpha_k)$ with $1,n\in \{\alpha_1,\ldots,\alpha_k\}$.  
The set $\{\alpha_1,\alpha_2,\ldots,\alpha_k\}$ has $k!=(n-1)!$ possible permutations, so $F^{1}_{(n-1)\times 1}(P_n)=\underbrace{P_n \cup \cdots\cup P_n}_{(n-1)!}$. 
\end{proof}
\begin{example}
   The supertoken graph $F^{1}_{3\times 1}(P_4)$ is shown in Figure \ref{fig:example-of-path}.
\end{example}
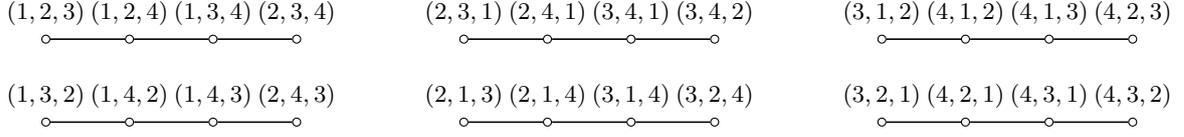
\begin{figure}   
    \begin{center}
    \begin{tikzpicture}[scale=1.1,auto,swap]
    \vertex (123) at (0,0) [label=above:{{\footnotesize$(1,2,3)$}}]{};
    \vertex (124) at (1,0) [label=above:{{\footnotesize $(1,2,4)$}}]{};
    \vertex (134) at (2,0) [label=above:{{\footnotesize $(1,3,4)$}}]{};
    \vertex (234) at (3,0) [label=above:{{\footnotesize $(2,3,4)$}}]{};
    \draw [line width=0.6pt] (234)--(134);
    \draw [line width=0.6pt] (124)--(134);
    \draw [line width=0.6pt] (124)--(123);
    \vertex (231) at (5,0) [label=above:{{\footnotesize$(2,3,1)$}}]{};
    \vertex (241) at (6,0) [label=above:{{\footnotesize $(2,4,1)$}}]{};
    \vertex (341) at (7,0) [label=above:{{\footnotesize $(3,4,1)$}}]{};
    \vertex (342) at (8,0) [label=above:{{\footnotesize $(3,4,2)$}}]{};
    \draw [line width=0.6pt] (342)--(341);
    \draw [line width=0.6pt] (341)--(241);
    \draw [line width=0.6pt] (241)--(231);
    \vertex (312) at (10,0) [label=above:{{\footnotesize$(3,1,2)$}}]{};
    \vertex (412) at (11,0) [label=above:{{\footnotesize $(4,1,2)$}}]{};
    \vertex (413) at (12,0) [label=above:{{\footnotesize $(4,1,3)$}}]{};
    \vertex (423) at (13,0) [label=above:{{\footnotesize $(4,2,3)$}}]{};
    \draw [line width=0.6pt] (423)--(413);
    \draw [line width=0.6pt] (412)--(413);
    \draw [line width=0.6pt] (412)--(312);
    \vertex (132) at (0,-1) [label=above:{{\footnotesize$(1,3,2)$}}]{};
    \vertex (142) at (1,-1) [label=above:{{\footnotesize $(1,4,2)$}}]{};
    \vertex (143) at (2,-1) [label=above:{{\footnotesize $(1,4,3)$}}]{};
    \vertex (243) at (3,-1) [label=above:{{\footnotesize $(2,4,3)$}}]{};
    \draw [line width=0.6pt] (243)--(143);
    \draw [line width=0.6pt] (143)--(142);
    \draw [line width=0.6pt] (142)--(132);
    \vertex (213) at (5,-1) [label=above:{{\footnotesize$(2,1,3)$}}]{};
    \vertex (214) at (6,-1) [label=above:{{\footnotesize $(2,1,4)$}}]{};
    \vertex (314) at (7,-1) [label=above:{{\footnotesize $(3,1,4)$}}]{};
    \vertex (324) at (8,-1) [label=above:{{\footnotesize $(3,2,4)$}}]{};
    \draw [line width=0.6pt] (324)--(314);
    \draw [line width=0.6pt] (214)--(314);
    \draw [line width=0.6pt] (214)--(213);
    \vertex (321) at (10,-1) [label=above:{{\footnotesize$(3,2,1)$}}]{};
    \vertex (421) at (11,-1) [label=above:{{\footnotesize $(4,2,1)$}}]{};
    \vertex (431) at (12,-1) [label=above:{{\footnotesize $(4,3,1)$}}]{};
    \vertex (432) at (13,-1) [label=above:{{\footnotesize $(4,3,2)$}}]{};
    \draw [line width=0.6pt] (432)--(431);
    \draw [line width=0.6pt] (421)--(431);
    \draw [line width=0.6pt] (421)--(321);
    \end{tikzpicture}
    \caption{The supertoken graph $F^{1}_{3\times 1}(P_4)$. 
    }
     \label{fig:example-of-path}
    \end{center}
\end{figure}
\begin{proposition}\label{prop:cycle-cycle}
    Let $G=C_n$ be a cycle with order $n$. Then, $F^{1}_{k\times 1}(G)$ contains at least one cycle for $k\le n-1$, and does not contain any cycle for $k=n$. Moreover, for $k=n$, the supertoken graph $F^{1}_{k\times 1}(G)$ is the graph with $n!$ isolated vertices and $F^{1}_{k\times 1}(G)=\underbrace{C_{n(n-1)} \cup \cdots\cup C_{n(n-1)}}_{(n-2)!}$ for $k=n-1$.
\end{proposition}
\begin{proof}
    The case $k=n$ is trivial, and the case $k\le n-2$ is similar to Proposition \ref{prop:path-cycle}.
    Consider $k=n-1$. 
    Denote $V(C_n)=\{1,2,\ldots,n\}$. 
    Let $\sigma$ be a permutation on $\{1,\ldots,n\}$ such that $\sigma(i)=i+1~(\text{mod}~n)$.
    For a vertex $(\alpha_1,\alpha_2,\ldots,\alpha_k)$ in $F^{1}_{k\times 1}(C_n)$  
    satisfying that $\{\alpha_1,\alpha_2,\ldots,\alpha_k\}$ is a permutation on $\{2,3,\ldots,n\}$ with $\alpha_j=n$, there is no vertex being on its right from the definition shown in (\ref{eq:increase}). Then, we define the relation $(\alpha'_1,\ldots,\alpha'_k)\succeq (\alpha_1,\ldots,\alpha_k)$ if $\alpha'_j=1$ and $\alpha'_l=\alpha_l$ for $l\in [k]\backslash\{j\}$, which implies that $(\alpha_1,\ldots,\sigma(\alpha_j),\ldots,\alpha_k)\succeq (\alpha_1,\ldots,\alpha_j,\ldots,\alpha_k)$. 

    Then, for any vertex $(\alpha_1,\alpha_2,\ldots,\alpha_k)$ in $F^{1}_{k\times 1}(C_n)$, there exists a vertex $(\alpha'_1,\alpha'_2,\ldots,\alpha'_k)$ on its right satisfying $(\alpha'_1,\alpha'_2,\ldots,\alpha'_k)$
    $\succeq (\alpha_1,\alpha_2,\ldots,\alpha_k)$. Moreover, the path from $(\alpha_1,\alpha_2$ $,\ldots,\alpha_k)$ to $(\sigma(\alpha_1),\sigma(\alpha_2),\ldots,\sigma(\alpha_k))$ is of length $k=n-1$.
    Note that $(\sigma^n(\alpha_1),\sigma^n(\alpha_2),\ldots,$ $\sigma^n(\alpha_k))=(\alpha_1,\alpha_2,\ldots,\alpha_k)$.
    Thus, there is a cycle of length $n(n-1)$ including $(\alpha_1,\alpha_2,\ldots,\alpha_k)$.
    Observe that the set $\{\alpha_1,\alpha_2,\ldots,\alpha_k\}$ has $k!=(n-1)!$ possible permutations and $k$ of them are in the same cycle of length $n(n-1)$. Thus, $F^{1}_{
    (n-1)\times 1}(C_n)=\underbrace{C_{n(n-1)} \cup \cdots\cup C_{n(n-1)}}_{(n-2)!}$.
\end{proof}
\begin{example}
    The supertoken graph $F^{1}_{3\times 1}(C_4)$ is shown in Figure \ref{fig:example-of-cycle}.
\end{example}
    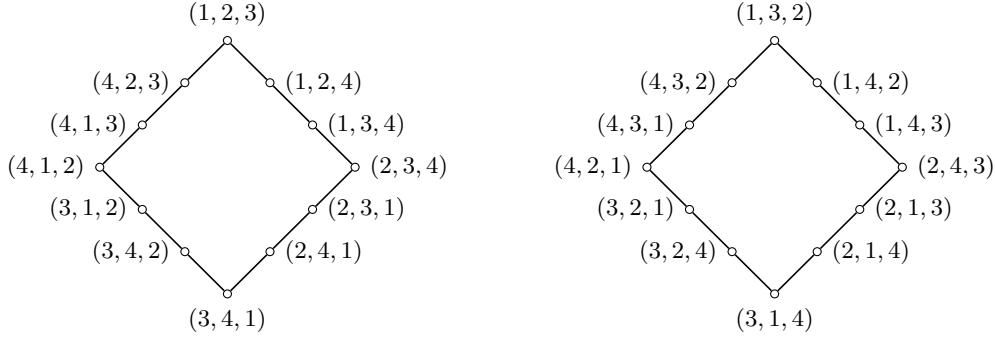
\begin{figure}   
    \begin{center}
    \begin{tikzpicture}[scale=0.8,auto,swap]
    \vertex (123) at (0,2.1) [label=above:{{\footnotesize$(1,2,3)$}}]{};
    \vertex (124) at (0.7,1.4) [label=right:{{\footnotesize $(1,2,4)$}}]{};
    \vertex (134) at (1.4,0.7) [label=right:{{\footnotesize $(1,3,4)$}}]{};
    \vertex (234) at (2.1,0) [label=right:{{\footnotesize $(2,3,4)$}}]{};
    \vertex (231) at (1.4,-0.7) [label=right:{{\footnotesize$(2,3,1)$}}]{};
    \vertex (241) at (0.7,-1.4) [label=right:{{\footnotesize $(2,4,1)$}}]{};
    \vertex (341) at (0,-2.1) [label=below:{{\footnotesize $(3,4,1)$}}]{};
    \vertex (342) at (-0.7,-1.4) [label=left:{{\footnotesize $(3,4,2)$}}]{};
    \vertex (312) at (-1.4,-0.7) [label=left:{{\footnotesize$(3,1,2)$}}]{};
    \vertex (412) at (-2.1,0) [label=left:{{\footnotesize $(4,1,2)$}}]{};
    \vertex (413) at (-1.4,0.7) [label=left:{{\footnotesize $(4,1,3)$}}]{};
    \vertex (423) at (-0.7,1.4) [label=left:{{\footnotesize $(4,2,3)$}}]{};
    \draw [line width=0.6pt] (123)--(124);
    \draw [line width=0.6pt] (124)--(134);
    \draw [line width=0.6pt] (134)--(234);
    \draw [line width=0.6pt] (234)--(231);
    \draw [line width=0.6pt] (231)--(241);
    \draw [line width=0.6pt] (241)--(341);
    \draw [line width=0.6pt] (341)--(342);
    \draw [line width=0.6pt] (342)--(312);
    \draw [line width=0.6pt] (312)--(412);
    \draw [line width=0.6pt] (412)--(413);
    \draw [line width=0.6pt] (423)--(413);
    \draw [line width=0.6pt] (423)--(123);
    \vertex (132) at (9,2.1) [label=above:{{\footnotesize$(1,3,2)$}}]{};
    \vertex (142) at (9.7,1.4) [label=right:{{\footnotesize $(1,4,2)$}}]{};
    \vertex (143) at (10.4,0.7) [label=right:{{\footnotesize $(1,4,3)$}}]{};
    \vertex (243) at (11.1,0) [label=right:{{\footnotesize $(2,4,3)$}}]{};
    \vertex (213) at (10.4,-0.7) [label=right:{{\footnotesize$(2,1,3)$}}]{};
    \vertex (214) at (9.7,-1.4) [label=right:{{\footnotesize $(2,1,4)$}}]{};
    \vertex (314) at (9,-2.1) [label=below:{{\footnotesize $(3,1,4)$}}]{};
    \vertex (324) at (8.3,-1.4) [label=left:{{\footnotesize $(3,2,4)$}}]{};
    \vertex (321) at (7.6,-0.7) [label=left:{{\footnotesize$(3,2,1)$}}]{};
    \vertex (421) at (6.9,0) [label=left:{{\footnotesize $(4,2,1)$}}]{};
    \vertex (431) at (7.6,0.7) [label=left:{{\footnotesize $(4,3,1)$}}]{};
    \vertex (432) at (8.3,1.4) [label=left:{{\footnotesize $(4,3,2)$}}]{};
    \draw [line width=0.6pt] (132)--(142);
    \draw [line width=0.6pt] (142)--(143);
    \draw [line width=0.6pt] (143)--(243);
    \draw [line width=0.6pt] (243)--(213);
    \draw [line width=0.6pt] (213)--(214);
    \draw [line width=0.6pt] (214)--(314);
    \draw [line width=0.6pt] (314)--(324);
    \draw [line width=0.6pt] (324)--(321);
    \draw [line width=0.6pt] (321)--(421);
    \draw [line width=0.6pt] (421)--(431);
    \draw [line width=0.6pt] (431)--(432);
    \draw [line width=0.6pt] (432)--(132);
    \end{tikzpicture}
    \caption{The supertoken graph $F^{1}_{3\times 1}(C_4)$. 
    }
     \label{fig:example-of-cycle}
    \end{center}
\end{figure}

\section{The supertoken graphs $F^{s}_{k\times 1}(G)$}
\label{sec:F_k^s(G)-colors-dif}

In this case, the $k$ tokens are different, and the maximum number of tokens per vertex is $s$, with $1<s<k$.
For integers $t_1,\ldots,t_i$ satisfying $t_1+\cdots+t_i\le k$, we define $\binom{k}{t_1,\ldots,t_i}=\frac{k!}{t_{1}!\cdots t_{i}!(k-t_{1}-\cdots-t_{i})!}$.\\
\begin{proposition}
Let $G=(V,E)$ be a graph on $n$ vertices. The order and size of the supertoken graph $F^{s}_{k\times 1}(G)$ are, respectively,
\begin{eqnarray}
|V(F^{s}_{k\times 1}(G))| &= &\sum_{j=0}^{\lfloor\frac{k}{s+1}\rfloor}(-1)^j\binom{n}{j}\sum_{\substack{t_{i_1},t_{i_2},\ldots,t_{i_j}\ge s+1\\ t_{i_1}+t_{i_2}+\cdots+t_{i_j}\le k}}\binom{k}{t_{i_1},\ldots,t_{i_j}}(n-j)^{k-t_{i_1}-\cdots-t_{i_j}}\label{eq:nomial2}\\
&:=& h(n,k,s),\notag
\end{eqnarray}
and 
\begin{equation}\label{eq:edge-fks2}
|E(F_k^s(G))| = |E|\sum_{\substack{0\le r_u,r_v\le s-1\\r_u+r_v\le k-1}}(r_u+1)\binom{k}{r_u+1,r_v}h(n-2,k-r_u-r_v-1,s).
\end{equation}
Specifically, for $j=0$, we define the second sum $\sum_{\substack{t_{i_1},t_{i_2},\ldots,t_{i_j}\ge s+1\\ t_{i_1}+t_{i_2}+\cdots+t_{i_j}\le k}}\binom{k}{t_{i_1},\ldots,t_{i_j}}$ to be equal to $1$ in (\ref{eq:nomial2}), and for $r_u+r_v=k-1$, we define $h(n,0,s)$ to be $1$ in (\ref{eq:edge-fks2}).
\end{proposition}
\begin{proof}
    Let $(\alpha_1,\alpha_2,\ldots,\alpha_k)$ be a vertex of $F^{k}_{k\times 1}(G)$. We get $|V(F^{k}_{k\times 1}(G))|=n^k$.
    Suppose that $j$ is the number of distinct vertices appearing at least $s+1$ times in $\alpha_1,\alpha_2,\ldots,\alpha_k$. 
    Then, $j\le {\lfloor\frac{k}{s+1}\rfloor}$. 
    Let $P_i\subset V(F^{k}_{k\times 1}(G))$ be the set where at least $s+1$ of $\alpha_1,\alpha_2,\ldots,\alpha_k$ are equal to vertex $i\in V$. 
    Then,
    \begin{equation*}
        |P_i|=\sum_{t_i=s+1}^k\binom{k}{t_i}(n-1)^{k-t_i}.
    \end{equation*}
    Moreover, we have 
    \begin{equation*}
        |P_{i_1}\cap \cdots \cap P_{i_j}|=\sum_{\substack{t_{i_1},t_{i_2},\ldots,t_{i_j}\ge s+1\\ t_{i_1}+t_{i_2}+\cdots+t_{i_j}\le k}}\binom{k}{t_{i_1},\ldots,t_{i_j}}(n-j)^{k-t_{i_1}-\cdots-t_{i_j}}
    \end{equation*}
    for $j=1,\ldots,{\lfloor\frac{k}{s+1}\rfloor}$.
    By the inclusion-exclusion principle, we find that 
    \begin{eqnarray*}
       && |V(F^{s}_{k\times 1}(G))|=|V(F^{k}_{k\times 1}(G))|-\left|P_{i_1}\cup \cdots \cup P_{i_{\lfloor\frac{k}{s+1}\rfloor}}\right|\\
        &=&|V(F^{k}_{k\times 1}(G))|-\sum|P_i|+\sum|P_{i_1}\cap P_{i_2}|+\cdots\\
        &+&(-1)^{\lfloor\frac{k}{s+1}\rfloor}\sum\left|P_{i_1}\cap P_{i_2} \cap \cdots \cap P_{i_{\lfloor\frac{k}{s+1}\rfloor}}\right|\\
        &=&\sum_{j=0}^{\lfloor\frac{k}{s+1}\rfloor}(-1)^j\binom{n}{j}\sum_{\substack{t_{i_1},t_{i_2},\ldots,t_{i_j}\ge s+1\\ t_{i_1}+t_{i_2}+\cdots+t_{i_j}\le k}}\binom{k}{t_{i_1},\ldots,t_{i_j}}(n-j)^{k-t_{i_1}-\cdots-t_{i_j}}.
    \end{eqnarray*}
    Consider one edge $uv\in E$. 
    Then, the edges corresponding to $uv$ in $F^{s}_{k\times 1}(G)$ are 
    \begin{align*}
        \{A_uA_v:&A_u=(\alpha_1,\ldots,\alpha_i=u,\ldots,\alpha_k),A_v=(\alpha_1,\ldots,\alpha_i=v,\ldots,\alpha_k)\\
        &\mbox{with}~\alpha_j\in V~\mbox{for}~j\in [k]\backslash\{i\}\}.
    \end{align*}
    Let $r_u$ and $r_v$ be the numbers of elements equal to $u$ and $v$ in $\{\alpha_1,\ldots,\alpha_k\}\backslash\{\alpha_i\}$. 
    It follows that $0\le r_u,r_v\le s-1$ and $r_u+r_v\le k-1$.
    We get the number of edges corresponding to $uv$ in $F^{s}_{k\times 1}(G)$ is 
    \begin{equation}\label{eq:edge-fk1s}
    \sum_{\substack{0\le r_u,r_v\le s-1\\r_u+r_v\le k-1}}(r_u+1)\binom{k}{r_u+1,r_v}h(n-2,k-r_u-r_v-1,s).    
    \end{equation}
    Therefore, the equation (\ref{eq:edge-fks2}) holds.
\end{proof}

\begin{example}
    Consider a star $S_4$ (with 5 vertices) with vertex set $V(S_4)=\{0,1,2,3,4\}$, where $0$ is the central vertex. The order of the supertoken graph $F^2_{3\times1}(S_4)$ is $h(5,3,2)=5^3+(-1)^{1}\binom{5}{1}\binom{3}{3}(5-1)^{3-3}=120$. 
    We show the edges in $F^2_{3\times1}(S_4)$ corresponding to the edge $01$ in $G$. 
    Let $A=(\alpha_1,\alpha_2,\alpha_3)\in V(F^2_{3\times1}(S_4))$ and $B=(\beta_1,\beta_2,\beta_3)\in V(F^2_{3\times1}(S_4))$ such that $\alpha_i=0$ and $\beta_i=1$ for some $i\in [3]$ and $\alpha_j=\beta_j$ for $j \in [3]\backslash\{i\}$. 
    We consider all possible cases of $A$ and $B$.
    Let $r_0$ and $r_1$ be the numbers of elements equal to $0$ and $1$ in 
    $\{\alpha_1,\alpha_2,\alpha_3\}\backslash\{\alpha_i\}$.
    Since $s=2$, it follows that $0\le r_0,r_1\le 1$. Then, 
    there are four possible combinations of values for $r_0$ and $r_1$.
    
    {\bf Case 1.} $r_0=0$ and $r_1=0$.
    Here, the value $i$ can be $1,2$, or $3$, that is, there are $3$ possible values for $i$. 
    For the remaining elements in $\{\alpha_1,\alpha_2,\alpha_3\}\backslash\{\alpha_i\}$,
    the possible combinations are 
    \begin{equation*}
        \{2,3\},\{3,2\},\{2,4\},\{4,2\},\{3,4\},\{4,3\},\{2,2\},\{3,3\},\{4,4\}.
    \end{equation*}
    In this case, there are $3\times9=27$ edges corresponding to $01$, see Table \ref{table:case1}.
    \begin{table}[t]
\begin{center}
\begin{tabular}{|c|c|c|c|c|c|c|c|c|c|}
  \hline
  $(0,2,3)$&$(0,3,2)$   &$(0,2,4)$&$(0,4,2)$&$(0,3,4)$&$(0,4,3)$&$(0,2,2)$&$(0,3,3)$&$(0,4,4)$ \\
  $(1,2,3)$&$(1,3,2)$   &$(1,2,4)$&$(1,4,2)$&$(1,3,4)$&$(1,4,3)$&$(1,2,2)$&$(1,3,3)$&$(1,4,4)$\\
  \hline
  $(2,0,3)$&$(3,0,2)$   &$(2,0,4)$&$(4,0,2)$&$(3,0,4)$&$(4,0,3)$&$(2,0,2)$&$(3,0,3)$&$(4,0,4)$ \\
  $(2,1,3)$&$(3,1,2)$   &$(2,1,4)$&$(4,1,2)$&$(3,1,4)$&$(4,1,3)$&$(2,1,2)$&$(3,1,3)$&$(4,1,4)$\\
  \hline
  $(2,3,0)$&$(3,2,0)$   &$(2,4,0)$&$(4,2,0)$&$(3,4,0)$&$(4,3,0)$&$(2,2,0)$&$(3,3,0)$&$(4,4,0)$ \\
  $(2,3,1)$&$(3,2,1)$   &$(2,4,1)$&$(4,2,1)$&$(3,4,1)$&$(4,3,1)$&$(2,2,1)$&$(3,3,1)$&$(4,4,1)$\\
  \hline
\end{tabular}
\caption{Edges in $F^2_{3\times1}(S_4)$ corresponding to $01$ for $r_0=0$ and $r_1=0$. 
}
\label{table:case1}
\end{center}
\end{table}

 {\bf Case 2.} $r_0=0$ and $r_1=1$.
    Here, the value $i$ can be $1,2$, or $3$, and there exist $j\in [3]\backslash \{i\}$ such that $\alpha_{j}=1$. 
    So, there are $6$ possible ordered pairs for $(i,j)$, which are $(1,2), (1,3), (2,3), (2,1), (3,1)$, and $(3,2)$. 
    For the remaining elements in $\{\alpha_1,\alpha_2,\alpha_3\}\backslash\{\alpha_{i},\alpha_{j}\}$,
    the possible combinations are $\{2\},\{3\},\{4\}$.
    In this case, there are $6\times3=18$ edges corresponding to $01$, see Table \ref{table:case2}.
    \begin{table}[t]
\begin{center}
\begin{tabular}{|c|c|c|c|c|c|c|c|c|c|}
  \hline
  $(0,1,2)$&$(0,1,3)$   &$(0,1,4)$&$(0,2,1)$&$(0,3,1)$&$(0,4,1)$&$(2,0,1)$&$(3,0,1)$&$(4,0,1)$ \\
  $(1,1,2)$&$(1,1,3)$   &$(1,1,4)$&$(1,2,1)$&$(1,3,1)$&$(1,4,1)$&$(2,1,1)$&$(3,1,1)$&$(4,1,1)$\\
  \hline
  $(1,0,2)$&$(1,0,3)$   &$(1,0,4)$&$(1,2,0)$&$(1,3,0)$&$(1,4,0)$&$(2,1,0)$&$(3,1,0)$&$(4,1,0)$ \\
  $(1,1,2)$&$(1,1,3)$   &$(1,1,4)$&$(1,2,1)$&$(1,3,1)$&$(1,4,1)$&$(2,1,1)$&$(3,1,1)$&$(4,1,1)$\\
  \hline
\end{tabular}
\caption{Edges in $F^2_{3\times1}(S_4)$ corresponding to $01$ for $r_0=0$ and $r_1=1$. 
}
\label{table:case2}
\end{center}
\end{table}

 {\bf Case 3.} $r_0=1$ and $r_1=0$.
    There exist $i_1,i_2\in [3]$ such that $\alpha_{i_1}=\alpha_{i_2}=0$.
    Then, there are $3$ possible ordered pairs for $(i_1,i_2)$, which are $(1,2), (1,3)$, and $(2,3)$ (here, $(i_2,i_1)=(2,1)$ is the same as $(i_2,i_1)=(1,2)$). 
    For the remaining elements in $\{\alpha_1,\alpha_2,\alpha_3\}\backslash\{\alpha_{i_1},\alpha_{i_2}\}$,
    the possible combinations are $\{2\},\{3\},\{4\}$.
    Moreover, for $\alpha_{i_1}=\alpha_{i_2}=0$, we find that there exists $j\in\{i_1,i_2\}$ satisfying $\beta_{j}=1$ and $\beta_{l}=\alpha_{l}$ for $l\in \{i_1,i_2\}\backslash\{j\}$.
    Thus, $j$ can be $i_1$ or $i_2$, which implies that there are $2$ values for $j$.
    In this case, there are $2\times3\times3=18$ edges corresponding to $01$, see Table \ref{table:case3}.
    \begin{table}[t]
\begin{center}
\begin{tabular}{|c|c|c|c|c|c|c|c|c|c|}
  \hline
  $(0,0,2)$&$(0,0,3)$   &$(0,0,4)$&$(0,2,0)$&$(0,3,0)$&$(0,4,0)$&$(2,0,0)$&$(3,0,0)$&$(4,0,0)$ \\
  $(1,0,2)$&$(1,0,3)$   &$(1,0,4)$&$(1,2,0)$&$(1,3,0)$&$(1,4,0)$&$(2,1,0)$&$(3,1,0)$&$(4,1,0)$\\
  \hline
   $(0,0,2)$&$(0,0,3)$   &$(0,0,4)$&$(0,2,0)$&$(0,3,0)$&$(0,4,0)$&$(2,0,0)$&$(3,0,0)$&$(4,0,0)$ \\
   $(0,1,2)$&$(0,1,3)$   &$(0,1,4)$&$(0,2,1)$&$(0,3,1)$&$(0,4,1)$&$(2,0,1)$&$(3,0,1)$&$(4,0,1)$ \\
  \hline
\end{tabular}
\caption{Edges in $F^2_{3\times1}(S_4)$ corresponding to $01$ for $r_0=1$ and $r_1=0$. 
}
\label{table:case3}
\end{center}
\end{table}

 {\bf Case 4.} $r_0=1$ and $r_1=1$.
    Here, $\alpha_{i_1}=\alpha_{i_2}=0$ and $\alpha_{i_3}=1$. 
    Then, there are $3$ possible ordered pairs of $(i_1,i_2,i_3)$, which are $(1,2,3), (1,3,2)$, and $(2,3,1)$. 
    Moreover, for $\alpha_{i_1}=\alpha_{i_2}=0$, we see that there exists $j\in \{i_1,i_2\}$ satisfying $\beta_{j}=1$ and $\beta_{l}=\alpha_{l}$ for $l\in \{i_1,i_2\}\backslash\{j\}$.
    Thus, $j$ can be $i_1$ or $i_2$, which implies that there are $2$ values for $j$.
    In this case, there are $2\times3=6$ edges corresponding to $01$, see Table \ref{table:case4}.
    Therefore, the number of edges in $F^2_{3\times1}(S_4)$ corresponding to $01$ is $27+18+18+6=69$.
    \begin{table}[t]
\begin{center}
\begin{tabular}{|c|c|c|c|c|c|c|c|c|c|}
  \hline
  $(0,0,1)$&$(0,1,0)$   &$(1,0,0)$&$(0,0,1)$&$(0,1,0)$   &$(1,0,0)$ \\
  $(1,0,1)$&$(1,1,0)$   &$(1,1,0)$&$(0,1,1)$&$(0,1,1)$   &$(1,0,1)$ \\ 
  \hline
\end{tabular}
\caption{Edges in $F^2_{3\times1}(S_4)$ corresponding to $01$ for $r_0=1$ and $r_1=1$. }
\label{table:case4}
\end{center}
\end{table}

Alternatively, from (\ref{eq:edge-fk1s}), we have 
    \begin{eqnarray*}
        &&\sum_{r_0=0}^{1}\sum_{r_1=0}^{1}(r_0+1)\binom{3}{r_0+1,r_1}h(3,3-r_0-r_1-1,2)\\
        &=&1\cdot\binom{3}{1}\cdot h(3,2,2)+1\cdot\binom{3}{1,1}\cdot h(3,1,2)+2\cdot \binom{3}{2,0}\cdot h(3,1,2)\\
\end{eqnarray*}
         \begin{equation*}
        +2\cdot \binom{3}{2,1}\cdot h(3,0,2)=27+18+18+6=69.
    \end{equation*}
\end{example}

\section{The Cartesian product $F_{k\times 1}^k(G)\cong G\Box\stackrel{(k)}{\cdots}\Box G$}
\label{sec:prod-cartesia}

In this case, the $k$ tokens are different, and the maximum number of tokens per vertex is $k$.
The supertoken graph $F^{k}_{k\times 1}(G)\cong G\Box\stackrel{(k)}{\cdots}\Box G$ is the Cartesian $k$-th product of $G$ by itself. Then, $F^{k}_{k\times 1}(G)$ has order $n^k$ and size $mkn^{k-1}$.
See an example in Figure~\ref{C_4+2-supertokens(C_4)}$(e)$ for the case $G=C_4$ ($n=4$) and 2 tokens of different colors.

\begin{proposition}\label{prop:cartesian}
If $G=(V,E)$ is a (not necessarily regular) graph, then the supertoken graph $F^{k}_{k\times 1}(G)$ is isomorphic to the Cartesian product $G\Box\stackrel{(k)}{\cdots}\Box G$.
\end{proposition}

\begin{proof}
Let $F^k=F^{k}_{k\times 1}(G)$ and $G^k=G\Box\stackrel{(k)}{\cdots}\Box G$. Each vertex of $F^k$ can be represented by a vector $(u_1,u_2,\ldots,u_k)$, where $u_i\in V$ is the vertex with token $i\in\{1,\ldots,k\}$. This vertex is adjacent to the vertices
\begin{equation}
\begin{split}
& (v_1,u_2,\ldots,u_k) \mbox{ with } v_1\sim u_1 \mbox{ (token 1 moved from } u_1 \mbox{ to } v_1),\\
& (u_1,v_2,\ldots,u_k) \mbox{ with } v_2\sim u_2 \mbox{ (token 2 moved from } u_2 \mbox{ to } v_2),\\
& \hskip1.25cm \vdots\\
& (u_1,u_2,\ldots,v_k) \mbox{ with } v_k\sim u_k \mbox{ (token $k$ moved from } u_k \mbox{ to } v_k).
\end{split}
\label{eq:vertexs-adj}
\end{equation}
Besides, each vertex of $G^k$ is also represented by a vector $(u_1,u_2,\ldots,u_k)$, where $u_i$ is a vertex of the $i$-th factor (copy of $G$) of the product $G\Box\stackrel{(k)}{\cdots}\Box G$, and its adjacent vertices are as in~\eqref{eq:vertexs-adj}. This proves the claimed isomorphism.
\end{proof}

\end{document}